\documentclass{article}
\usepackage[utf8]{inputenc}
\usepackage{mathrsfs}
\usepackage{authblk}
\usepackage{mathtools}
\usepackage[colorinlistoftodos]{todonotes}  

\usepackage{comment}
\usepackage{amsmath}
\usepackage{amssymb}
\usepackage{amsthm}
\usepackage{mathrsfs}

\numberwithin{equation}{section}

\usepackage[left=3cm, right=3cm]{geometry}

\usepackage[T1]{fontenc}
\usepackage[utf8]{inputenc}
\usepackage[english]{babel}
  \usepackage{lmodern}  
\usepackage{graphicx}
\usepackage{subfigure}
\usepackage{tikz}
\usepackage{makeidx}
\makeindex

\usepackage{multicol}
\usepackage{float}
\usepackage{tikz,pgf}
\usepackage{tikz-cd}
\usepackage{float}
\usepackage{xcolor}
\usepackage[all,knot,arc,color,web]{xy}
\usepackage{graphicx}
\xyoption{curve}
\usepackage{bbold}
\usepackage{fancyhdr}
\usepackage{stmaryrd}
\usepackage{hyperref}
\theoremstyle{plain}
\newtheorem{theorem}{Theorem}[section]
\newtheorem*{theorem*}{Theorem}
\newtheorem{proposition}[theorem]{Proposition}
\newtheorem*{proposition*}{Proposition}
\newtheorem{corollary}[theorem]{Corollary}
\newtheorem{lemma}[theorem]{Lemma}
\newtheorem{definition}[theorem]{Definition}

\theoremstyle{definition}
\newtheorem{remark}[theorem]{Remark}

\newtheorem{notation}[theorem]{Notation}

\newtheorem*{theoremA}{Theorem A}
\newtheorem*{theoremB}{Theorem B}
\newtheorem*{theoremC}{Theorem C}
\newtheorem*{theoremD}{Theorem D}

\newcommand{\luca}[1]{\todo[size=\tiny, backgroundcolor=green, linecolor=green]{LUCA: #1}}

\newcommand{\olga}[1]{\todo[size=\tiny]{OLGA: #1}}

 \title{Generic properties of planar symplectic billiards}
\author[1]{Luca Baracco}
\author[1]{Olga Bernardi}
\author[2]{Jos\'e Pedro Gaivão}
\affil[1]{Dipartimento di Matematica Tullio Levi--Civita, Università di Padova, Italy}
\affil[2]{ISEG, Universidade de Lisboa, Portugal}
\date{\today}

\begin{document}

\maketitle

\abstract{\noindent We study generic properties of planar symplectic billiards, a symplectic analogue of classical Birkhoff billiards 
introduced by P. Albers and S. Tabachnikov. We prove that, for a residual set of $C^\infty$ strongly convex domains, periodic symplectic billiard trajectories are in general position and non-degenerate. We also prove a Franks' lemma for symplectic billiards and deduce that stably hyperbolic periodic points form a hyperbolic set given by the closure of the set of hyperbolic periodic points. We further show that, for a residual set of $C^\infty$ strongly convex domains, elliptic periodic points are stable and hyperbolic periodic points have transverse homoclinic points. Based on these results, we conclude that --generically-- $C^\infty$ strongly convex symplectic billiard has positive topological entropy.}


\section{Introduction}

\noindent The study of billiards occupies a central place in dynamical systems, geometry, and mathematical physics. Given a bounded domain $\Omega\subset\mathbb{R}^2$ with 
boundary $\Gamma=\partial\Omega$, the billiard flow describes the motion of a point particle moving with unit speed inside $\Omega$ and undergoing elastic reflections at the boundary according to the law of equal angles.  Recording only successive reflections, we obtain the so-called billiard map, an exact area-preserving twist map encoding the essential dynamical information of the billiard flow, see e.g. \cite{Sergei} and \cite[Chapter 3]{Siburg}. Many fundamental questions in billiard dynamics, including the existence and stability of periodic trajectories, invariant curves, hyperbolic sets, and positive entropy, can be formulated in terms of the billiard map. Since the pioneering works of Birkhoff, billiards have provided a rich source of examples exhibiting a wide range of dynamical phenomena, ranging from complete integrability to hyperbolicity and chaos. 
In recent years, substantial progress has been achieved in understanding the generic properties of convex billiards \cite{lazutkin1981convex,Gruber1990}, including generic properties of periodic orbits \cite{Sto87,PeSto,PetkovStojanov1987b,CARNEIROKAMPHORSTCARVALHO2003}, transverse intersections of invariant manifolds \cite{Donnay2005,DCOKPC07}, and positive topological entropy for generic convex billiards \cite{Cheng2004,BDDGT24}. \\
~\newline
\noindent Symplectic billiards, introduced in 2018 by P. Albers and S. Tabachnikov \cite{AT}, provide a symplectic counterpart of the classical Birkhoff billiard. Whereas the classical billiard map is generated by the Euclidean length functional, the symplectic billiard map is generated by the symplectic area functional. An equivalent description of symplectic billiards, with a more geometric flavour, is the following: given any three consecutive points $x,y,z\in\Gamma$ of a symplectic billiard trajectory, the tangent line to $\Gamma$ at $y$ is parallel to the chord $xz$. 
In \cite{AT} Albers and Tabachnikov developed the basic theory of symplectic billiards, establishing their variational framework, proving the existence of periodic orbits, deriving an analogue of Lazutkin's invariant curve theorem, and revealing both similarities and differences with respect to classical Birkhoff billiards. Since then, several questions regarding symplectic billiards have attracted attention. Rigidity and integrability questions have been studied in analogy with the classical Birkhoff conjecture; see, for example, \cite{BaBe}, \cite{BBN2}, \cite{BBN1}, \cite{Tsod}, \cite{BBN3} and \cite{FSV}. Moreover, several questions related to polygonal symplectic billiards have also studied in \cite{Albers}, inspired by numerical investigations, and connections with Minkowski billiards have been exploited in \cite{ACT26}. Finally, in \cite{BBFN}, the authors have addressed the attention to dissipative symplectic billiards in terms both of the strength of dissipation and of the shape of the table. All these developments suggest that symplectic billiards provide a rich class in which one may revisit the classical questions of billiard dynamics. In contrast, generic dynamical properties of symplectic billiards remain largely unexplored. \\
~\newline
\noindent The purpose of this paper is to initiate a systematic study of generic dynamics in planar symplectic billiards. By a generic property, we mean a property that holds on a residual subset of strongly convex domains endowed with the $C^\infty$ topology. Our guiding principle is that many of the fundamental generic properties known for classical billiards should admit symplectic counterparts, although their proofs require substantial and non-trivial modifications due to the different reflection law and variational formulation. \\
~\newline
\noindent After a brief introduction to symplectic billiards and recalling some basic facts in Section \ref{INTRO DEF}, we start our study with periodic orbits. In the classical billiard setting, generic positions of periodic orbits were investigated by V. Petkov and L. Stojanov (see \cite{Sto87} and \cite{PeSto}), who proved that, generically, periodic billiard trajectories have ``zero defect'' --that is do not pass through the same reflection point two or more times-- and different periodic orbits do not share reflection points. The core of Section \ref{section 3} is proving the symplectic analogues of Petkov and Stojanov's generic theorems: 
\begin{theoremA}
\textit{$(a)$ For a residual set of smooth strongly convex domains, no periodic symplectic billiard orbit passes --within one period-- through the same point more than once. \\
$(b)$ For a residual set of smooth strongly convex domains, each pair of distinct periodic symplectic billiard orbits have no common reflection points.}  
\end{theoremA}
\noindent We refer to Theorems~\ref{T1} and Theorem~\ref{T2} respectively. The proofs combine the variational characterization of periodic orbits with the Multijet Transversality Theorem. \\
With similar techniques, in Section \ref{section 4} we establish generic non-degeneracy of periodic symplectic billiard orbits. More precisely, the main results of this section are the next ones.
\begin{theoremB}
\textit{$(a)$ For a residual set of smooth strongly convex domains, the corresponding symplectic billiard dynamics admits only non-resonant elliptic periodic points. \\
$(b)$ For a residual set of smooth strongly convex domains, the corresponding symplectic billiard dynamics admits only a finite number of periodic orbits for each period $n \ge 3$ and they are all non-degenerate. \\
$(c)$ For a residual set of smooth strongly convex domains, the corresponding symplectic billiard dynamics admits only periodic orbit with at least one non self-conjugate point.}   
\end{theoremB}
\noindent See respectively Theorem~\ref{nonresonant}, Corollary ~\ref{finite and nondeg} and Theorem~\ref{T4}. These results are the starting point for the perturbative analysis developed later in the paper and may be regarded as a symplectic counterpart of the corresponding genericity statements known for classical billiards. \\
~\newline
\noindent Another classical result in generic perturbations of dynamical systems is the so-called Franks' lemma. Originally introduced by J. Franks in the context of diffeomorphisms \cite[Lemma 1.1]{franks71}, this lemma allows to realize prescribed perturbations of the derivative along a finite orbit by small perturbations of the diffeomorphism itself. Analogous results have proved extremely useful in conservative and Hamiltonian dynamics \cite{AlishahDias14,Buzzi_2017}. In Section \ref{section 5} we establish a Franks' lemma for planar symplectic billiards: 
\begin{theoremC}
\textit{There exists a residual set of smooth strongly convex domains such that, for every periodic orbit $\mathcal{O}$ inside one of these domains $\Gamma$ of period at least $3$ and for every $C^2$ neighborhood $U$ of $\Gamma$, there exists an open ball $B \subset Sp(1)$ centered in the linearized dynamics along $\mathcal{O}$ such that any element of $B$ is realizable as the linearized dynamics around $\tilde{\Gamma}$ for some $\tilde{\Gamma} \in U$. Moreover, such a perturbation can be supported in an arbitrarily small neighborhood of three successive bouncing points of $\mathcal{O}$.}
\end{theoremC} 
\noindent We  refer to Theorem~\ref{Franks} for details. Such a result has been obtained for classical planar billiards in \cite{Viss}. See also \cite[Theorem 6.1]{BDDGT24} for the multidimensional version. Interestingly, we remark that --unlike the classical billiard case-- the Franks' lemma for symplectic billiards does not require a no-focusing condition along three-link segments of periodic trajectories. This leads to a considerably simpler version.
Using Franks' lemma for symplectic billiards, in Theorem~\ref{stably hyperbolicity}, we show that symplectic billiards with stably hyperbolic periodic points have a hyperbolic set as the closure of the set of hyperbolic periodic points. \\
~\newline
\noindent Subsequently, in Section \ref{Jose}, we study the dynamics in a neighborhood of elliptic periodic points and prove that elliptic islands occur generically in planar symplectic billiards (Theorem~\ref{marzo}). This phenomenon is familiar from the theory of area-preserving twist maps and constitutes an important obstruction to global hyperbolicity. Similar results were obtained for classical billiards in \cite{CARNEIROKAMPHORSTCARVALHO2003,BG10}. \\
~\newline
\noindent In the final part of the paper, we first show how to make transverse heteroclinic intersections between invariant manifolds of hyperbolic periodic orbits (Proposition~\ref{prop:transversality}). This result, combined with the generic properties of periodic orbits, constitutes a version of the Kupka-Smale theorem for symplectic billiards and provides a geometric mechanism for the creation of chaotic dynamics. In particular, we consequently show that, generically, every hyperbolic periodic point has a transverse homoclinic point (Theorem~\ref{transverse homoclinic}). As a corollary, we prove the next result. 
\begin{theoremD}
\textit{For a residual set of smooth strongly convex domains, the corresponding symplectic billiard dynamics has positive topological entropy.} 
\end{theoremD}
\noindent We refer to Corollary~\ref{positive entropy}. This result may be viewed as a symplectic counterpart of the generic positive topological entropy results established for classical convex billiards \cite{Cheng2004,BDDGT24}. See also the recent preprint \cite{FLORIO} showing that positive topological entropy is generic in the analytic category. \\
~\newline
All together, the results of this paper establish a generic theory of planar symplectic billiards parallel to that developed over the past decades for classical billiards, showing that many of the fundamental mechanisms responsible for complex behavior in classical billiards translate to the symplectic setting.

\section{Symplectic billiards} \label{INTRO DEF}
\noindent Let $\Omega \subset (\mathbb{R}^2,\omega)$ be a strictly convex planar domain with $C^k$ boundary $\Gamma$, $k\geq 2$. Assume that 
the perimeter of $\Gamma$ is normalized to $2 \pi$. Fix the origin $O \in \mathrm{int}(\Omega)$ and the positive counter-clockwise orientation on $\Gamma$. Let $\gamma: \mathbb{S} \to \Gamma$ be a $C^k$ parametrization of $\Gamma$ such that $\gamma(\mathbb{S})=\Gamma$, where $\mathbb{S} = \mathbb{R}/2\pi\mathbb{Z}$. Given $x_1,x_2 \in \mathbb{R}^2$,
\[
\omega\colon \mathbb{R}^2 \times \mathbb{R}^2  \ni (x_1,x_2)\mapsto \omega(x_1,x_2):=\det(x_1-O,x_2 - O)\in\mathbb{R}\, ,
\]
see Figure \ref{L}. In particular, for all $t_1,t_2\in\mathbb{S}$, the notation \[
\omega(\gamma(t_1),\gamma(t_2))
\]
indicates the signed area of the parallelogram of sides $\gamma(t_1)-O$ and $\gamma(t_2)-O$. 
\begin{figure}[ht]
 \centering
\includegraphics[scale=0.35]{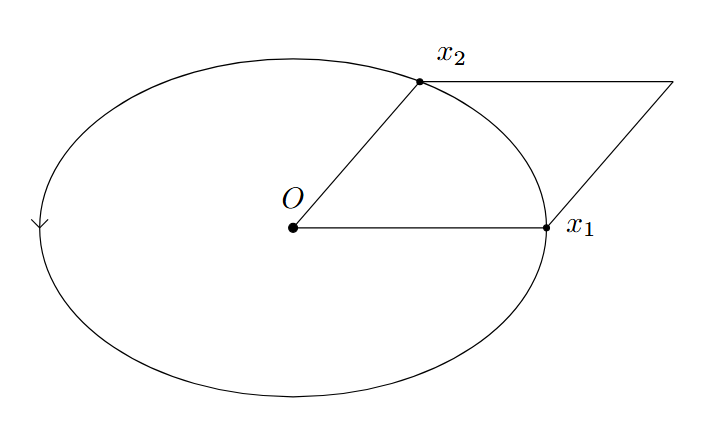}
    \caption{$\det(x_1-O,x_2-O)$ is the area of the parallelogram in figure.}
    \label{L}
\end{figure}
\\

\noindent Since $\Omega$ is strictly convex, for every point $\gamma(t) \in \Gamma$, there exists a unique point $\gamma(t^*)$ ($t^* \ne t$) such that 
\[
\omega(\gamma'(t),\gamma'(t^*)) = 0\, .
\]
(In the above formula, $\gamma'(t)$ and $\gamma'(t^*)$ denote the tangent vectors at $\Gamma$ at the points $\gamma(t)$ and $\gamma(t^*)$ respectively). We refer to
\begin{equation} \label{PS}
\hat{\mathcal{P}}_\gamma = \{ (\gamma(t_1),\gamma(t_2)) \in \Gamma \times \Gamma: \ \gamma(t_1) < \gamma(t_2) < \gamma(t^*_1)\}
\end{equation}
as the (open, positive) phase space, see \cite[Section 2.2]{BBFN} for details. We define the symplectic billiard map on $\hat{\mathcal{P}}_\gamma$ as follows (see \cite[Page 5]{AT}):
$$\hat{T}_\gamma: \hat{\mathcal{P}}_\gamma \to \hat{\mathcal{P}}_\gamma, \qquad (\gamma(t_1), \gamma(t_2) \mapsto (\gamma(t_2),\gamma(t_3))$$
where $\gamma(t_3) \in \Gamma$ is the unique point such that
$$\omega( \gamma^{\prime} (t_2),\gamma(t_3)-\gamma(t_1))=0\, ,$$ 
\noindent see Figure \ref{SB}. The next proposition summarizes the main properties of the symplectic billiard map. We refer to \cite[Section 2]{AT} for further details and for the proofs.

\begin{figure}[ht]
\centering
\includegraphics[scale=0.4]{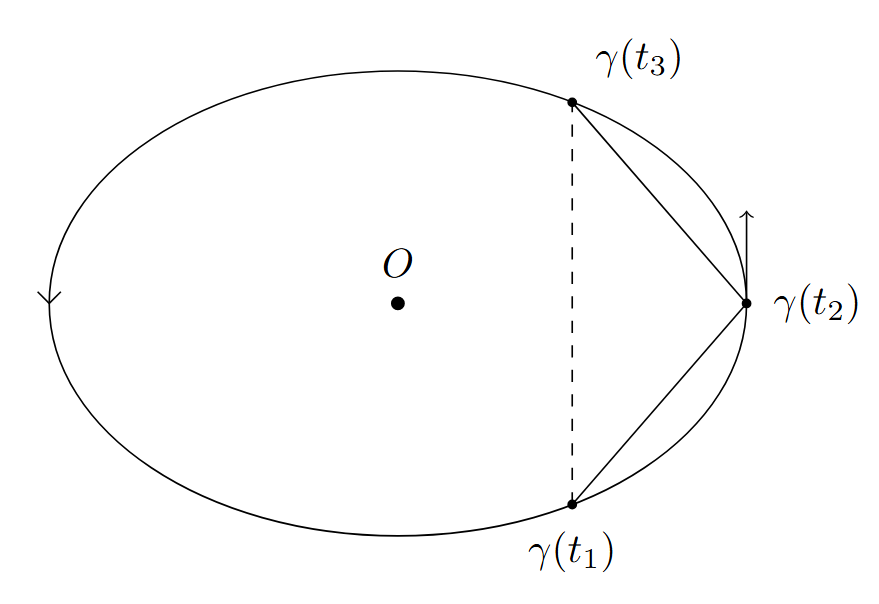}
    \caption{The symplectic billiard map reflection.}
    \label{SB}
\end{figure}

\begin{proposition}\label{properties}
The following properties hold.
\begin{enumerate}
    \item 
    $\hat{T}_\gamma$ is $C^{k-1}$ and it extends continuously to the closure of $\hat{\mathcal{P}}_\gamma$ so that 
$$\hat{T}_\gamma(\gamma(t),\gamma(t)) = ((\gamma(t),\gamma(t)) \qquad \text{and} \qquad \hat{T}_\gamma(\gamma(t),\gamma(t^*)) = ((\gamma(t^*),\gamma(t))\, .$$
\item For every $(\gamma(t_1),\gamma(t_2))\in\hat{\mathcal{P}}_\gamma$ one has
\begin{equation} \label{GF}
\hat{T}_\gamma(\gamma(t_1),\gamma(t_2)) = (\gamma(t_2),\gamma(t_3)) \qquad \Longleftrightarrow \qquad \omega(\gamma(t_1),\gamma'(t_2)) + \omega(\gamma'(t_2),\gamma(t_3)) = 0\, .
\end{equation}
In other words, $\omega$ is a generating function for $\hat{T}_\gamma$.
\item The map $\hat T_\gamma$ does not depend on the choice of the origin $O$. 
\item\label{punto 4 trsnf affini} The map $\hat T_\gamma$ commutes with any map obtained as affine transformation of the plane, since they preserve tangent directions. 
\end{enumerate}
\end{proposition}
\noindent From now on we use the notation 
$$L(t_1,t_2) := \omega(\gamma(t_1),\gamma(t_2)), \qquad L_1(t_1,t_2) := \omega(\gamma'(t_1),\gamma(t_2)) \qquad \text{and} \qquad L_2(t_1,t_2) := \omega(\gamma(t_1),\gamma'(t_2))\, .$$
Up to a change of coordinates, it is possible to see the symplectic billiard map as a (negative) twist map. Indeed, let
\begin{equation*}\label{change coordinates}
s_1 = -L_1(t_1,t_2), \qquad s_2 = L_2(t_1,t_2)
\end{equation*}
and 
$$\mathcal{P}_\gamma=\{ (t,s)\in\mathbb{S}\times\mathbb{R} :\ s\in (\psi_1(t),\psi_2(t) )\}\, ,$$
where 
\[\psi_1(t):=-L_1(t,t^*) < 0 \qquad \text{and} \qquad \psi_2(t):=-L_1(t,t) > 0\, ,\]
\noindent see Figure \ref{fig:cilindrosbilenco}. Then $(t,s)$ are coordinates on $\mathcal{P}_\gamma$ and we denote the symplectic billiard map on $\mathcal{P}_\gamma$ simply by $T_\gamma$:
$$T_\gamma: \mathcal{P}_\gamma \to \mathcal{P}_\gamma, \qquad (t_1,s_1 = -L_1(t_1,t_2)) \mapsto (t_2, s_2 = L_2(t_1,t_2))\, .$$
Moreover, $T_\gamma$ is a (negative) twist map: 
$$\frac{\partial s_1}{\partial t_2} = -L_{12}(t_1,t_2) < 0\, ,$$
which preserves $ds \wedge dt$.
\begin{figure}
    \centering
    \includegraphics[scale=0.15]{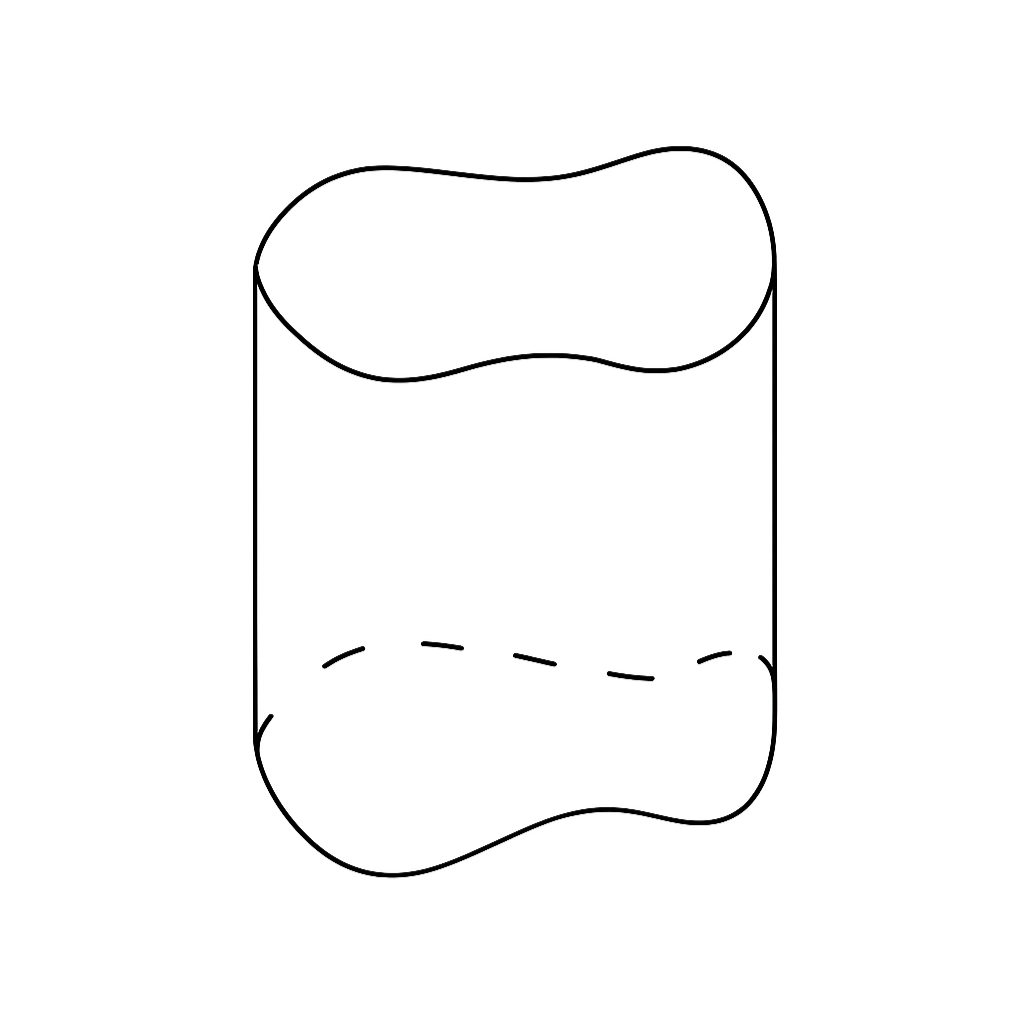}
    \caption{The phase space $\mathcal{P}_\gamma$.}
    \label{fig:cilindrosbilenco}
\end{figure} \\
\noindent The following lemma provides a formula for the differential of $T_\gamma$, which will be useful in many occasions later on. We refer to \cite[Lemma 2.18]{BBFN} for the proof. 
\begin{lemma}\label{formula DT lambda}
    Let $T_\gamma$ be the symplectic billiard map on $\mathcal{P}_\gamma$. For $(t_1,s_1 = -L_1(t_1,t_2)) \in \mathcal{P}_\gamma$, it holds:
    \begin{equation}
\label{differential}
DT_\gamma (t_1,s_1) = -\frac{1}{L_{12}(t_1,t_2)}\begin{pmatrix}
      L_{11}(t_1,t_2) & 1\\
- L_{12}^2(t_1,t_2)+ L_{22}(t_1,t_2)\cdot L_{11}(t_1,t_2)&  L_{22}(t_1,t_2)
    \end{pmatrix}\, .
\end{equation}
\end{lemma}

\begin{notation}
Through the whole paper, we shall assume that $\Omega$ is strongly convex that is the boundary $\Gamma$ of $\Omega$ has everywhere strictly positive curvature. The set of strongly convex billiard tables with $C^\infty$ boundary will be denoted by $\mathcal{S}$. Any element in $\mathcal{S}$ is defined by a parametrization $\gamma:\mathbb{S} \to \Gamma$ of its boundary so that we indicate $\gamma \in \mathcal{S}$ an element in $\mathcal{S}$. We endow $\mathcal{S}$ with the Whitney $C^\infty$ topology so that $\gamma_n\to \gamma$ iff $\|\gamma_n - \gamma\|_{C^k} \to 0$ for every $k \ge 0$. The space $\mathcal{S}$ equipped with the topology induced by the norms $\| \cdot \|_k$, for every $k \ge 0$, is Fréchet and also Baire. 
As a consequence, we have that every countable intersection of open and dense subsets in $\mathcal{S}$ is dense. Finally, in what follows, an element $\gamma \in \mathcal{S}$ is said to be generic if it is on a residual set of $\mathcal{S}$, i.e. $\gamma$ lays in a countable intersection of dense subsets of $\mathcal{S}$.   
\end{notation}


\section{Generic position of symplectic periodic orbits} \label{section 3}

\noindent This section is devoted to prove the following two theorems:
\begin{theorem} \label{T1}
There exists a residual set $\mathcal{T}_1 \subset \mathcal{S}$ such that for every $\gamma\in \mathcal{T}_1$ no periodic orbit in $\gamma$ passes --within one period-- through the same point more than once.
\end{theorem}
\begin{theorem}\label{T2}
There exists a residual set $\mathcal{T}_2 \subset \mathcal{S}$ such that for every $\gamma\in \mathcal{T}_2$ each
pair of distinct periodic orbits in $\gamma$ have no common reflection points.
\end{theorem} 

\noindent For the proof of these theorems, we shall follow the technique of \cite{Sto87} which, in turn, applies the Multijet Transversality Theorem (see \cite[Theorem 4.13 and Lemma 4.14]{GG73}). To this end, we need to premise some notations and definitions. \\
\noindent Given $n,s \in \mathbb{N}$, with $n \ge s$, let
$$\sigma: \{1,\ldots,n\}\rightarrow \{1,\ldots,s\}, \quad i \mapsto \sigma(i)$$
be a surjective map satisfying
\begin{enumerate}
\item[$(a)$] $\sigma(i) \neq \sigma(i+1)$ for every $i \in \{1,\ldots,n\}$;
\item[$(b)$] $\{\sigma(i),\sigma(i+1)\} \neq \{\sigma(j),\sigma(j+1) \}$ for every $i \ne j$.
\end{enumerate}
The map $\sigma$ will be employed to describe a $n$-periodic orbit in a given symplectic billiard table $\gamma \in \mathcal{S}$. If such an orbit passes --in one period-- through some reflection point more than once, then $n > s$. Otherwise, if every point of the orbit is visited --in one period-- exactly once, then $s = n$ and $\sigma$ is a bijection (in such a case, we can assume $\sigma = \text{id}$). In what follows, we still indicate by $\sigma$ the natural $n$-periodic extension to this map to $\mathbb{Z}$, that is
$$\sigma: \mathbb{Z} \rightarrow \{1,\ldots,s\}, \quad i \mapsto \sigma(i)$$
satisfying
\begin{enumerate}
\item[$(c)$] $\sigma(i + mn) = \sigma(i)$ for every $(i,m) \in \mathbb{Z} \times \mathbb{Z}$.
\end{enumerate}
Given $\mathcal{S} \ni \gamma: \mathbb{S} \to \Gamma$, the $1$-jet of $\gamma$ at a point $t \in \mathbb{S}$ is the $5$-vector:
$$j^1\gamma(t) := (t,\gamma(t),\gamma'(t))\, .$$
More generally, if $\gamma^s := \gamma \times \ldots \times \gamma: \mathbb{S} \times \ldots \times \mathbb{S} \to \Gamma \times \ldots \Gamma$ is the Cartesian product of $s$ copies of $\gamma$, the $1$-jet of $\gamma^s$ at a point $(t_1,\ldots,t_s) \in \mathbb{S}\times \ldots \times \mathbb{S}$ is the $5s$-vector:
\begin{eqnarray*}
j^1\gamma^s(t_1,\ldots,t_s) & := & \left(j^1\gamma(t_1),\ldots,j^1\gamma(t_s)\right) \\
& \cong &\left(t_1,\ldots,t_s,\gamma(t_1),\ldots,\gamma(t_s),\gamma'(t_1),\ldots,\gamma'(t_s)\right)\, .
\end{eqnarray*}
In order to give the next definition, for a set $A$, we indicate (as in \cite{Sto87}):
$$A^{(s)} := \{(a_1,\ldots,a_s) \in A^s: a_i \ne a_j \textnormal{ whenever } i \ne j\}\, .$$
\begin{definition} \label{def 1}
For $\sigma$ as above and $\gamma \in \mathcal{S}$, we set
$$\mathbb{S}^{(s)}_{\sigma,\gamma} := \{(t_1,\dots,t_s) \in \mathbb{S}^{(s)}: (\gamma(t_{\sigma(i)}),\gamma(t_{\sigma(i+1)})) \in \hat{\mathcal{P}}_\gamma \text{ for every } 1\le i\le n\}\, .$$
Moreover, for $\sigma$ as above, we set
$$
M_\sigma := \{j^1\gamma^s(t_1,\ldots,t_s): \gamma\in\mathcal{S} \text{ and } (t_1,\ldots,t_s) \in \mathbb{S}^{(s)}_{\sigma,\gamma}\}\, .
$$
\end{definition}
\noindent The set $\mathbb{S}^{(s)}_{\sigma,\gamma}$ is introduced to represent the natural domain of parameters on which the point sequence $\{\gamma(t_{\sigma(i)})\}_{i \in \mathbb{Z}}$ in $\Gamma$ constitutes an $n$-periodic orbit in $\gamma$. We refer to such an orbit as a $n$-periodic orbit of type $\sigma$. \\
In order to define the action corresponding to such an orbit, in the sequel we will use the notation explained right below. Let us denote by $(t_1,..,t_s,x_1,..,x_s,v_1,..,v_s)$ the coordinates in the $s$-fold bundle of $1$-jets $J^1_s(\mathbb{S},\mathbb{R}^2)$ so that, for some $j^1\gamma^s(t_1,\ldots,t_s) \in M_\sigma \subset J^1_s(\mathbb{S},\mathbb{R}^2)$, we have:
$$x_i = \gamma(t_i), \quad v_i = \gamma'(t_i)$$ 
for every $i \in \{1,\ldots,s\}$. $M_\sigma$ is an open submanifold of $J^1_s(\mathbb{S},\mathbb{R}^2)$ of dimension $5s$. \\

\noindent Let $n,s \in \mathbb{N}$, with $n \ge s$. For $\sigma$ as above and $(x_1, \ldots, x_s) \in \mathbb{R}^{2s}$,
\begin{equation} \label{elle sigma}
L_{\sigma}(x_1, \ldots , x_s) := \sum_{i = 1}^n \omega(x_{\sigma(i)},x_{\sigma(i+1)})\, ,\end{equation}
so that
\[ L_\sigma(\gamma(t_1),\ldots,\gamma(t_s)) := \sum_{i=1}^n \omega(\gamma(t_{\sigma(i)}),\gamma(t_{\sigma(i+1)}))\,.\]
The following proposition underlies the proof of Theorem \ref{T1}.  
\begin{proposition}\label{p1}
Let $n,s \in \mathbb{N}$, with $n > s$. For $\sigma$ as above there exists a submanifold $\Sigma_{\sigma} \subset J^1_s(\mathbb{S},\mathbb{R}^2)$ of codimension greater or equal than $s+1$ such that if  $\gamma\in \mathcal{S}$ and $\{\gamma(t_{\sigma(i)})\}_{i \in \mathbb{N}}$ is an $n$-periodic orbit in $\gamma$ then $j^1\gamma^s(t_1,\ldots,t_s) \in \Sigma_{\sigma}$.    
\end{proposition}
\begin{proof}
Since $\{\gamma(t_{\sigma(i)})\}_{i \in \mathbb{N}}$ is an orbit, it is necessarily critical for $L_{\sigma}$, that is:
\[ \partial_{t_i} L_{\sigma}(\gamma(t_1),...,\gamma(t_s)) = \sum_{j\in \sigma^{-1}(\{i\})} \omega(\gamma'(t_i),\gamma(t_{\sigma(j+1)})-\gamma(t_{\sigma(j-1)}))= 0 \qquad \forall i \in \{1,\ldots,s\}\, .\] 
Moreover, since by assumption $n > s$, there exists at least one index $i \in \{1,\ldots,s\}$  such that $\sigma^{-1}(\{i\}) = \{j_1,...,j_r\}$ for $r\ge 2$. Assume that $i=1$. Then, by the symplectic billiards dynamics, we have:
\[ \omega(\gamma'(t_1),\gamma(t_{\sigma(j_h +1)}) - \gamma(t_{\sigma(j_h -1)}))=0 \qquad \forall h \in \{1,\ldots,r\}\, .\]
These last two formulas suggest to consider the next equations on $M_\sigma$:
$$f_i(t,x,v) := \sum_{j\in \sigma^{-1}(\{i\})} \omega(v_i,x_{\sigma(j+1)}-x_{\sigma(j-1)})= 0 \qquad i \in \{2,\ldots,s\}\, ,$$
and
$$g_h(t,x,v) := \omega(v_1,x_{\sigma(j_h + 1)} - x_{\sigma(j_h - 1)}) = 0 \qquad h \in \{1,\ldots,r\}\, .$$
We proceed to prove that $f_i,g_h$ are independent on $M_\sigma$. \\
For $P := j^1\gamma^s(t_1,\ldots,t_s) \in M_\sigma$, we identify $T_PM_\sigma$ with $\mathbb{R}^{5s}$ and we say that 
$V \in T_PM_\sigma$ lies in the $x_k$-plane (resp. $v_k$-plane) if all the components of $V$ are zero except ones corresponding to $x_k$ (resp. $v_k$). Then, for every $k \in \{2,\ldots,s\}$, we consider the vector $V_k$ in the $v_k$-plane given by $v_k = \gamma(t_k)-\gamma(t_k^*)$; and, for every $l \in \{1,\ldots,r\}$, we consider the vector $W_l$ in the $x_{\sigma(j_l +1)}$-plane given by $x_{\sigma(j_l +1)} = J \gamma'(t_1)$. \\
From the one hand, we obtain:
\begin{eqnarray*}
\partial_{V_k} f_i(t,x,v) &=& \delta_{k i} \sum_{j \in \sigma^{-1}(\{i\})} \omega(
\gamma(t_k)-\gamma(t_k^*),\gamma(t_{\sigma(j+1)}) - \gamma(t_{\sigma(j-1)}) \\
&=& \sum_{j \in \sigma^{-1}(\{i\})} \omega(
\gamma(t_i)-\gamma(t_i^*),\gamma(t_{\sigma(j+1)}) - \gamma(t_{\sigma(j-1)})\, .
\end{eqnarray*}
We observe that this last sum cannot be $0$. In fact, the points $\gamma(t_i)$ and $\gamma(t_i^*)$ divide $\gamma$ into two arcs. In particular, for any $j \in \sigma^{-1}(\{i\})$, $\gamma(t_{\sigma(j+1)})$ belongs to the arc joining counterclockwise $\gamma(t_i)$ and $\gamma(t_i^*)$ and $\gamma(t_{\sigma(j-1)})$ belongs to the arc joining counterclockwise $\gamma(t_i^*)$ and $\gamma(t_i)$. Consequently, $\gamma(t_{\sigma(j+1)}) - \gamma(t_{\sigma(j-1)})$ is always transverse to the vector $v_i = \gamma(t_i) - \gamma(t_i^*)$ and always pointing to the same side. Hence its orthogonal projection along the vector $Jv_i$ is always positive. This assures that $\partial_{V_k} f_i(t,x,v) \neq 0$ when $k =i$. Moreover, $\partial_{V_k} g_h(t,x,v) = 0$ since the functions \(g_h\) do not depend on the variables \(v_k\) for \( k\ge2\). \\
From the other hand, 
$$\partial_{W_l} g_h(t,x,v) = \delta_{lh}\omega(\gamma'(t_1), J\gamma'(t_1)) = \|\gamma'(t_1)\|^2 \ne 0\, .$$
From the independence of $f_i,g_h$ on $M_\sigma$, we obtain that 
$$\Sigma_\sigma := \left\{P \in M_\sigma \text{ such that } f_i(P) = 0 = g_h(P), \ \forall i \in \{2,\dots,s\} \text{ and } \forall h \in \{1, \dots, r\}\right\}$$ 
is the desired submanifold of \(M_\sigma\) of codimension $s-1+r \ge s+1$.
\end{proof}
\noindent From the previous proposition --combined with the Multijet Transversality Theorem (see \cite{GG73} at Page 57, Theorem 4.13 and Lemma 4.14)-- the proof of Theorem \ref{T1} easily follows.
\begin{proof}[Proof of Theorem \ref{T1}]
Let $n, s \in \mathbb{N}$, with $n > s$. We first prove that, for $\sigma$ as above, there exists a residual set $V_{\sigma} \subset \mathcal{S}$ such that if $\gamma\in V_\sigma$ then $\gamma$ has no $n$-periodic orbits of type $\sigma$. Let
\begin{eqnarray*}
J^1_s: \mathbb{S}^{(s)} \times \mathcal{S} &\to& J^1_s(\mathbb{S},\mathbb{R}^2) \\
\left((t_1,\ldots,t_s),\gamma\right) &\mapsto& j^1\gamma^s(t_1,\ldots,t_s)
\end{eqnarray*}
Fixed $\gamma \in \mathcal{S}$, $J^1_s(\cdot,\gamma) \subset J^1_s(\mathbb{S},\mathbb{R}^2)$ is an $s$-dimensional submanifold of $J^1_s(\mathbb{S},\mathbb{R}^2)$. 
For $\sigma$ as above, let $\Sigma_\sigma$ be the submanifold of Proposition \ref{p1}, which has codimension greater or equal than $s+1$. Choose a countable collection $\{\Sigma_{\sigma,m}\}_{m \in \mathbb{N}}$ of compact submanifolds of $\Sigma_\sigma$ such that $\bigcup_{m \in \mathbb{N}} \Sigma_{\sigma,m} = \Sigma_\sigma$. By the Multijet Transversality Theorem, for every $m \in \mathbb{N}$, there exists an open and dense subset $V_{\sigma,m} \subset \mathcal{S}$ such that $J^1_s(\cdot,\gamma)$ is transverse to $\Sigma_\sigma$ for every $\gamma \in V_{\sigma,m}$. Then define $V_\sigma := \bigcap_{m \in \mathbb{N}} V_{\sigma,m}$. The set $V_\sigma$ is a countable intersection of open and dense sets, that is $V_\sigma$ is a residual set, such that 
$$V_\sigma = \{\gamma \in \mathcal{S}: J^1_s(\cdot,\gamma) \pitchfork \Sigma_\sigma\}\, .$$
However, due to dimensional considerations, we have that  
$$V_\sigma = \{\gamma \in \mathcal{S}: J^1_s(\cdot,\gamma) \cap \Sigma_\sigma = \emptyset\}\, ,$$
which is equivalent to saying that every $\gamma \in V_\sigma$ has no periodic orbits of type $\sigma$. The conclusion is then obtained by setting \(V = \bigcap_{\sigma} V_\sigma\), where 
$\sigma$ corresponds to $n, s \in \mathbb{N}$, with $n > s$. This set \(V\) is residual, since it is a countable intersection of residual sets. 
\end{proof}

\noindent The same technique will be used to prove Theorem \ref{T2}, regarding pairs of distinct periodic orbits without common reflection points. \\ \\
\noindent Given $n_1, n_2, s \in \mathbb{N}$ with $n_1 + n_2 \ge s$, let
$$\sigma_j: \{1,\ldots,n_j\}\rightarrow \{1,\ldots,s\}, \quad i \mapsto \sigma_j(i), \qquad j = 1,2\, ,$$
be maps satisfying
\begin{enumerate}
\item[$(a)$] $\sigma_j(i) \neq \sigma_j(i+1)$ for every $i \in \{1,\ldots,n_j\}$ and $j = 1,2 $;
\item[$(b)$] $\{\sigma_1(i),\sigma_1(i+1)\} \neq \{\sigma_2(i'),\sigma_2(i'+1)\}$ for $1 \le i \le n_1$ and $1\le i' \le n_2$;
\item[$(c)$] $\text{Im}(\sigma_1) \cup \text{Im}(\sigma_2)=\{1,...,s\}$.
\end{enumerate}
As previously, for $j = 1,2$, we still indicate by $\sigma_j$ the natural $n_j$-periodic extension to this map to $\mathbb{Z}$. The maps $\sigma_1, \sigma_2$ describe respectively a $n_1$-periodic orbit and a $n_2$-periodic orbit in a given table $\gamma \in \mathcal{S}$. Since $\text{Im}(\sigma_1) \cup \text{Im}(\sigma_2)=\{1,...,s\}$, if these orbits have common reflection points then $n_1 + n_2 > s$. Otherwise, $n_1 + n_2 = s$. \\
We now introduce the following definition.
\begin{definition}
For $\sigma_1, \sigma_2$ as above and $\gamma \in \mathcal{S}$, we set 
$$\mathbb{S}^{(s)}_{\sigma_1,\sigma_2,\gamma} := \{(t_1,\dots,t_s) \in \mathbb{S}^{(s)}: (\gamma(t_{\sigma_j(i)}),\gamma(t_{\sigma_j(i+1)})) \in \hat{\mathcal{P}}_\gamma \text{ for every } 1\le i\le n_j \text{ and } j=1,2\}\, .$$
Moreover, for $\sigma_1, \sigma_2$ as above, we set
$$
M_{\sigma_1,\sigma_2} := \{j^1\gamma^s(t_1,\ldots,t_s): \gamma\in\mathcal{S} \text{ and } (t_1,\ldots,t_s) \in \mathbb{S}^{(s)}_{\sigma_1,\sigma_2,\gamma}\}\, .
$$
\end{definition}
\begin{proposition} Let $n_1, n_2, s \in \mathbb{N}$, with $n_1+n_2 > s$. 
For $\sigma_1, \sigma_2$ as above there exists a submanifold $\Sigma_{\sigma_1,\sigma_2} \subset J^1_s(\mathbb{S},\mathbb{R}^2)$ of codimension greater or equal than $s+1$ such that if $\gamma\in \mathcal{S}$ and, for $j =1,2$, $\{\gamma(t_{\sigma_j(i)})\}_{i \in \mathbb{N}}$ is an $n_j$-periodic orbit in $\gamma$ then $j^1\gamma^s(t_1,\ldots,t_s) \in \Sigma_{\sigma_1,\sigma_2}$.  
\end{proposition}
\begin{proof} Since $\{\gamma(t_{\sigma_1(i)})\}_{i \in \mathbb{N}}$ and $\{\gamma(t_{\sigma_2(i)})\}_{i \in \mathbb{N}}$ are orbits, $\partial_{t_i} L_{\sigma_1,\sigma_2}(\gamma(t_1),...,\gamma(t_s))=0$ for every $i=1,\ldots,s$, where
 \[ L_{\sigma_1,\sigma_2}(\gamma(t_1),\ldots,\gamma(t_s)) := \sum_{j=1,2} \sum_{i=1}^{n_j} \omega(\gamma(t_{\sigma_j(i)}),\gamma(t_{\sigma_j(i+1)}))\, .\]
That is
\[ \sum_{j=1,2} \sum_{h \in \sigma_j^{-1}(\{i\})} \omega(\gamma'(t_{i}),\gamma(t_{\sigma_j(h+1)})-\gamma(t_{\sigma_j(h-1)})) =0 \qquad \forall i \in \{1,\ldots,s\}\, .\]
Moreover, since by assumption $n_1 + n_2 > s$, there exists at least one index $i \in \{1,\ldots,s\}$, say $i=1$, such that, by the symplectic billiards dynamics, we have:
\[ \omega(\gamma'(t_1),\gamma(t_{\sigma_1(h+1)})-\gamma(t_{\sigma_1(h-1)})) = 0 \qquad \forall h \in\sigma^{-1}_1(1) \]
and
\[ \omega(\gamma'(t_1),\gamma(t_{\sigma_2(h+1)})-\gamma(t_{\sigma_2(h-1)})) = 0 \qquad \forall h \in\sigma^{-1}_2(1)\, .\]
Following the same strategy of the proof of Proposition \ref{p1}, we consider the next equations on $M_{\sigma_1,\sigma_2}$:
$$f_i(t,x,v) := \sum_{j=1,2} \sum_{h \in \sigma_j^{-1}(\{i\})} \omega(v_i,x_{\sigma_j(h+1)}-x_{\sigma_j(h - 1)}) =0 \qquad i \in \{2,\ldots,s\}\, ,$$
$$g_1(t,x,v) := \omega(v_1,x_{\sigma_1(h+1)} - x_{\sigma_1(h-1)})=0 \qquad h \in\sigma^{-1}_1(1)$$
and
$$g_2(t,x,v) := \omega(v_1,x_{\sigma_2(h+1)} - x_{\sigma_2(h-1)})=0 \qquad h \in\sigma^{-1}_2(1)\, .$$
By the same argument of the proof of Proposition \ref{p1}, we obtain that $f_i$, $g_1$ and $g_2$ are independent on $M_{\sigma_1,\sigma_2}$. Consequently:
$$\Sigma_{\sigma_1,\sigma_2} := \left\{P \in M_\sigma \text{ such that } f_i(P) = 0 = g_j(P), \ \forall i \in \{2,\dots,s\} \text{ and } \forall j = 1,2\right\}$$
is the desired submanifold of \(M_\sigma\) of codimension greater or equal than $s - 1 + 2 =  s + 1$.
\end{proof}
\begin{proof}[Proof of Theorem \ref{T2}] Same construction as in the proof of Theorem \ref{T1}.
\end{proof}

\section{Generic non-degeneracy of symplectic period orbits} \label{section 4}

In this section, we prove two results about non-degeneracy of symplectic periodic orbits in generic billiard tables: in particular, we prove that --generically-- symplectic periodic orbits are non-degenerate (Theorem \ref{Nondeg}) and have at least one non self-conjugate point (Theorem \ref{T4}). To enter into details, we need to recall some definitions. \\ \\
\noindent Thanks to the parametrization $\gamma: \mathbb{S} \to \Gamma$, it is possible to identify the tangent bundle $T\Gamma$ with $\Gamma \times \mathbb{R}$ by choosing the section $(\gamma(t),\gamma'(t))$ as a basis. Let $(t_i,s_i)_{i\in \mathbb{Z}}$ be a billiard orbit, which means  that $T_\gamma(t_i,s_i)=(t_{i+1},s_{i+1})$ for every $i \in \mathbb{Z}$ (as usual, for negative $i$ we mean the iterate of the inverse map $T_\gamma^{-1}$). 
\begin{definition} [Jacobi field and conjugate points] \label{JACOBI}
$(i)$ A sequence $\{\xi_i\}_{i \in \mathbb{Z}}$, $\xi_i \in T_{\gamma(t_i)} \Gamma$ is called a Jacobi field along $\{t_i\}_{i \in \mathbb{Z}}$ if it satisfies:
\begin{equation} \label{jacobi}
 L_{12}(t_{i-1}, t_i) \xi_{i-1} + [L_{22}(t_{i-1}, t_i) + L_{11}(t_{i}, t_{i+1})] \xi_i + L_{12}(t_{i}, t_{i+1}) \xi_{i+1} = 0   
\end{equation}
for all $i \in \mathbb{Z}$. \\
~\newline
\noindent 
$(ii)$ Let $M < N$. Two points $t_M$, $t_N$ of a symplectic billiard configuration are called conjugate if there is a non-zero Jacobi field vanishing at $t_M$ and $t_N$. 
\end{definition}
\noindent Jacobi fields are genuine vector fields defined along billiard trajectories. More precisely, given a trajectory $\{(t_i,s_i)\}_{i \in \mathbb{Z}}$, take $(\xi_1,\eta_1) \in T_{(t_1,s_1)} \mathcal{P}_\gamma$, a tangent vector to the phase space at the initial point $(t_1,s_1)$. Then propagate this vector through the billiard map by setting $(\xi_2,\eta_2):=DT_\gamma(t_1,s_1) \cdot (\xi_1,\eta_1)$ and inductively $(\xi_i,\eta_i):=DT_\gamma(t_{i-1},s_{i-1}) \cdot (\xi_{i-1},\eta_{i-1})$. The resulting sequence $\{\xi_i\}_{i \in \mathbb{Z}}$ defines a Jacobi field. Conversely, every Jacobi field is obtained in this way. \\ \\
In order to prove the results of this section, we need to introduce the $n$-fold bundle of $2$-jets $J^2_n(\mathbb{S},\mathbb{R}^2)$. Given $\mathcal{S} \ni \gamma: \mathbb{S} \to \Gamma$, the $2$-jet of $\gamma$ at a point $t \in \mathbb{S}$ is the $7$-vector:
$$j^2\gamma(t) := (t,\gamma(t),\gamma'(t),\gamma''(t))\, .$$
More generally:
\begin{eqnarray*}
j^2\gamma^n(t_1,\ldots,t_n) & := & \left(j^2\gamma(t_1),\ldots,j^2\gamma(t_n)\right) \\
& \cong &\left(t_1,\ldots,t_n,\gamma(t_1),\ldots,\gamma(t_n),\gamma'(t_1),\ldots,\gamma'(t_n),\gamma''(t_1),\ldots,\gamma''(t_n)\right)\, .
\end{eqnarray*}
\begin{definition} \label{def 1}
For $\gamma \in \mathcal{S}$, we set
$$\mathbb{S}^{(n)}_{\gamma} := \{(t_1,\dots,t_n) \in \mathbb{S}^{(n)}: (\gamma(t_{i}),\gamma(t_{i+1})) \in \hat{\mathcal{P}}_\gamma \text{ for every } 1\le i\le n\}\, .$$
Moreover:
$$
M := \{j^2\gamma^n(t_1,\ldots,t_n): \gamma\in\mathcal{S} \text{ and } (t_1,\ldots,t_n) \in \mathbb{S}^{(n)}_{\gamma}\}\, .
$$
\end{definition}
\noindent In the sequel we denote by $(t_1,..,t_n,x_1,..,x_n,v_1,..,v_n,a_1,...,a_n)$ the coordinates in the $n$-fold bundle of $2$-jets $J^2_n(\mathbb{S},\mathbb{R}^2)$ so that, for some $j^2\gamma^n(t_1,\ldots,t_n) \in M \subset J^2_n(\mathbb{S},\mathbb{R}^2)$, we have:
$$x_i = \gamma(t_i), \quad v_i = \gamma'(t_i), \quad a_i = \gamma''(t_i)$$ 
for every $i \in \{1,\ldots,n\}$. $M$ is an open submanifold of $J^2_n(\mathbb{S},\mathbb{R}^2)$ of dimension $7n$. \\ \\
\noindent For $\gamma \in \mathcal{S}$, let now $\mathcal{O} = \{ (t_1,s_1), (t_2,s_2), \ldots , (t_n,s_n) \}$ be a $n$-periodic orbit of $T_\gamma$, in particular $n \in \mathbb{N}$ is the smallest period of such a orbit. 
\begin{definition} [Nondegenerate periodic orbit]
We say that $\mathcal{O}$ is nondegenerate if $DT_\gamma^n(t_1,s_1)$ has no $1$ as an eigenvalue. 
\end{definition} 
\indent Applying the Multijet Transversality Theorem as for Theorems \ref{T1} and \ref{T2}, the next result easily follows.
\begin{theorem}\label{Nondeg}
There exists a residual set $\mathcal{T}_3 \subset \mathcal{S}$ such that every periodic orbit in $\gamma \in \mathcal{T}_3$ is non-degenerate.
\end{theorem}
\begin{proof} 
Let $n \ge 3$. We first notice that, since $\mathcal{O} = \{ (t_1,s_1), (t_2,s_2), \ldots , (t_n,s_n) \}$ is a $n$-periodic orbit, the dynamical conditions:
\begin{equation} \label{con2}
\omega(\gamma'(t_i),\gamma(t_{i+1}) - \gamma(t_{i-1})) = 0 \qquad \forall i \in \{1,\ldots,n\}
\end{equation}
are satisfied. Moreover, in terms of tangent vectors, this orbit is degenerate if there exists $(\xi_1,\eta_1) \in T_{(t_1,s_1)}\mathcal{P}_\gamma$ such that
\begin{equation} \label{eigenvec}
DT_\gamma^n(t_1,s_1) \cdot (\xi_1,\eta_1) = (\xi_1,\eta_1)\, .   
\end{equation} 
Let us extend $\mathcal{O}$ by periodicity, setting $(t_{i+nk},s_{i+nk}) = (t_i,s_i)$ for every integer $k$. Then consider the associated Jacobi field $\{\xi_i\}_{i \in \mathbb{Z}}$ defined by $(\xi_i,\eta_i) := DT_\gamma(t_{i-1},s_{i-1}) \cdot (\xi_{i-1},\eta_{i-1})$. By \eqref{eigenvec} this field is periodic with period $n$. Now by \eqref{jacobi} and by the periodicity condition we have that $\{\xi_i\}_{i \in \mathbb{N}}$ must satisfy the following linear system:
\begin{equation}\label{matrice}
\underbrace{\begin{pmatrix}
    B_1 & l_2 & 0 & \cdots & 0 & l_{1} \\
    l_2 & B_2 & l_3 & 0 & \cdots & 0 \\
    0 & l_3 & B_3 & l_4 & 0 & \cdots \\
    \vdots & \vdots & \vdots & \vdots & \vdots & \vdots \\
    l_1 & 0 & \cdots & 0 & l_{n} & B_n
\end{pmatrix}}_{:= \mathcal{H}}
\begin{pmatrix}
    \xi_1 \\
    \xi_2 \\
    \xi_3 \\
    \vdots \\
    \xi_n
\end{pmatrix}
=
\begin{pmatrix}
    0 \\
    0 \\
    0\\
    \vdots \\
    0
\end{pmatrix}
\end{equation}
where
$$
B_i=\omega(\gamma''(t_i),\gamma(t_{i+1})-\gamma(t_{i-1})) \quad \text{ and } \quad l_i=\omega( \gamma'(t_{i-1}),\gamma'(t_{i}))
$$
for every $i \in \{1, \ldots , n\}$ (note that, in writing the linear system, we took into account that $B_i$ and $l_i$, as well as $t_i$ and $\xi_i$, are $n$-periodic). Clearly we have that $\mathcal{O} = \{(t_1,s_1),...,(t_n,s_n)\}$ is degenerate if and only if 
\begin{equation} \label{con1}
\det \mathcal{H} = 0\, .
\end{equation}
We now proceed as in the proof of Theorems \ref{T1} and \ref{T2}. In particular, starting from conditions (\ref{con2}) and (\ref{con1}), we consider the next equations on $M$:
$$f_i(t,x,v,a) := \omega(v_i,x_{i+1}-x_{i-1}) = 0 \qquad i \in \{1,\ldots,n\}\, ,$$
and
$$g(t,x,v,a) := \det \mathcal{H} (t,x,v,a) = 0\, .$$
We remark that the $f_i$'s are independent on $M$ for the same argument as in the proof of Proposition \ref{p1}. Now, for $P := j^2\gamma^n(t_1,\ldots,t_n) \in M$, let 
$$\Sigma_n := \left\{P \in M \text{ such that } f_i(P) = 0 = g(P), \ \forall i \in \{1,\dots,n\}\right\}\, .$$ 
Since the equations defining $\Sigma_n$ are $C^\omega$, and the $f_i$'s are independent on $M$, we have that $\Sigma_n$ is a union of submanifolds $\Sigma^j$ of codimension $j \ge n$, see for instance Proposition 14.2.12 in \cite{Treves2022}. In more detail, we notice that this union extends over submanifolds of codimension $j > n$. This follows from the facts that the $f_i$'s are independent of $a$, while $g$ is a polynomial in the variables $a$ whose leading term --see formula (\ref{matrice})-- is $B_1 \cdots B_n=\prod_{i=1}^n \omega(a_i,x_{i+1}-x_{i-1})$, which never vanishes since $x_{i+1} \ne x_i$. Consequently:
$$\Sigma_n = \bigcup_{j=n+1}^{7n} \Sigma^j\, .$$
To complete the proof, we apply the Multijet Transversality Theorem with respect to each submanifold $\Sigma^j$, thus obtaining a family of residual sets $V_j$, whose intersection gives a set $V_n$ such that every $n$-periodic orbit in $\gamma \in V_n$ is
non-degenerate. Finally, by taking the countable intersection of residual sets $\bigcap_{n \ge 3} V_n$, we get the conclusion. 
\end{proof}
\begin{remark}
We stress that, through the previous proof, we restricted to the case where the periodic orbit passes –within one period– through the same point exactly once. In fact, if this case doesn't occur, the condition on the rank of the system of equations defining $\Sigma_n$ is satisfied without the need of the non-degeneracy condition. 
\end{remark}
\indent We proceed now to recall the notion of elliptic periodic point.  
\begin{definition} [Elliptic periodic point] \label{elliptic def} Given a $n$-period orbit $\mathcal{O} = \{ (t_1,s_1), (t_2,s_2), \ldots , (t_n,s_n) \}$ of $T_\gamma$, we say that $(t_1,s_1) \in \mathcal{P}_\gamma$ is an elliptic $n$-periodic point if $DT^n_\gamma(t_1,s_1)$ has eigenvalues $e^{\pm i \theta}$ for some $\theta \in (0,2\pi) \setminus \{\pi\}$. In particular, we say that $(t_1,s_1)$ is non-resonant if $q \theta \notin 2\pi\mathbb{Z}$ for $q=2,3,4,\dots$. 
\end{definition}
\noindent By \eqref{jacobi} and by the periodicity condition, we then have that $(t_1,s_1)$ is non-resonant if and only if
\[ \det
\underbrace{\begin{pmatrix}
    B_1 & l_2 & 0 & \cdots & 0 & l_{n}e^{\mp i\theta} \\
    l_2 & B_2 & l_3 & 0 & \cdots & 0 \\
    0 & l_3 & B_3 & l_4 & 0 & \cdots \\
    \vdots & \vdots & \vdots & \vdots & \vdots & \vdots \\
    l_1 e^{\pm i\theta}& 0 & \cdots & 0 & l_n & B_n
\end{pmatrix}}_{:= \mathcal{H}_\theta}
=0
\]
With the same arguments as for the proof of the previous Theorem \ref{Nondeg} (only adding the condition $\det\mathcal{H}_\theta = 0$ to $\det\mathcal{K} = 0$), we can then prove the following result.
\begin{theorem}\label{nonresonant}
There exists a residual set $\mathcal{T}_4 \subset \mathcal{S}$ such that every periodic orbit in $\gamma \in \mathcal{T}_4$ is non-degenerate and every elliptic periodic point is non-resonant.
\end{theorem}

\indent We remind now the next well-known fact on non-degenerate periodic orbits. 
\begin{proposition} \label{si sapeva gia}
Let $n \ge 3$ and $\gamma \in \mathcal{S}$. Then $T_\gamma$ admits a finite number of non-degenerate $n$-periodic orbits.
\end{proposition}
\begin{proof}
Let $n \ge 3$. As in \cite[Proposition 6]{DCOKPC07}, we first notice that every $n$-periodic orbit must have a point lying in a compact subset of the phase-cylinder. In fact, the existence for every $\epsilon>0$ of a $n$-periodic orbit entirely contained in $[0,2\pi)\times(0,\epsilon)\cup[0,2\pi)\times(\pi-\epsilon,\pi)$ would contradict the $n$-periodicity of the orbit itself. Consequently, corresponding to $n$, there exists $\varepsilon > 0$ such that every $n$-periodic orbit has at least one point in $[0,2\pi)\times[\epsilon,\pi-\epsilon]$. Since every $n$-periodic orbit is assumed to be non-degenerate, each of its points is isolated and then there can be only a finite number of them in a compact cylinder. This proves that $T_\gamma$ has only a finite number of non-degenerate $n$-periodic orbits.
\end{proof}
\noindent Theorem \ref{Nondeg} and Proposition \ref{si sapeva gia} then gives the next corollary.
\begin{corollary}\label{finite and nondeg}
For every $\gamma\in \mathcal{T}_3$, the corresponding symplectic billiard map $T_\gamma$ admits only a finite number of periodic orbits for each period $n\geq3$ and they are all non-degenerate.
\end{corollary}

\indent For the same $n$-period orbit $\mathcal{O} = \{ (t_1,s_1), (t_2,s_2), \ldots , (t_n,s_n) \}$, assume now that the point $t_1$ is self-conjugate, i.e. $t_1$ and $t_{n+1}$ (with $\gamma(t_{n+1}) = \gamma(t_1)$) are conjugate. By Definition \ref{JACOBI}-- this means that there exists a non-zero $n$-periodic Jacobi field $\{\xi_i\}_{i \in \mathbb{Z}}$ along $\{t_i\}_{i \in \mathbb{Z}}$ vanishing at $t_1$ and $t_{n+1}$. By \eqref{jacobi} and by the periodicity condition, we then have that $\{\xi_i\}_{i \in \mathbb{Z}}$ must satisfy the following linear system:
\begin{equation*}
\underbrace{\begin{pmatrix}
    B_2 & l_3 & 0 & \cdots & 0 \\
    l_3 & B_3 & l_4 & 0 & \cdots \\
    \vdots&\vdots&\vdots&\vdots&\vdots \\
    0 & \cdots & 0 & l_n & B_n
\end{pmatrix}}_{:=\mathcal{K}}
\begin{pmatrix}
    \zeta_2 \\
    \zeta_3 \\
    \vdots \\
    \zeta_{n}
\end{pmatrix}
=
\begin{pmatrix}
    0 \\
    0 \\
    \vdots \\
    0
\end{pmatrix}
\end{equation*}
for some $(\xi_2,...,\xi_n)\neq 0$. In other words, the configuration $t_1$ in $\mathcal{O} = \{(t_1,s_1),...,(t_n,s_n)\}$ is self-conjugate if and only if 
\begin{equation*} 
\det \mathcal{K} = 0\, .
\end{equation*}
With the same arguments as for the proof of the previous Theorem \ref{Nondeg} (only substituting the condition $\det\mathcal{H} = 0$ with $\det\mathcal{K} = 0$), we can then prove the following result.
\begin{theorem} \label{T4}
There exists a residual set $\mathcal{T}_5 \subset \mathcal{S}$ such that every periodic orbit in $\gamma \in \mathcal{T}_5$ has at least one non self-conjugate point.
  \end{theorem}

\section{Franks' lemma for generic symplectic billiards} \label{section 5}

\noindent This section contains the proof of Franks' lemma for symplectic billiards, saying that --for a residual set of strongly convex billiard tables-- an arbitrary perturbation of the derivative of the symplectic billiard map along a periodic orbit of period $\ge 3$ can be realized through a small perturbation of the strongly convex symplectic billiard table in the $C^{2}$ topology. The proof follows the same lines of the corresponding result for Birkhoff billiards, given in \cite[Theorem 1]{Viss}. \\
~\newline
\noindent Let $\gamma \in \mathcal{S}$. In what follows, given a $n$-periodic orbit $\mathcal{O} = \{(t_1,s_1), (t_2,s_2), \ldots, (t_n,s_n)\}$ for $T_\gamma$, $DT_\gamma(\mathcal{O})$ denotes the derivative of $T_\gamma$ along $\mathcal{O}$, that is:
$$DT_\gamma(\mathcal{O}) = DT_\gamma(t_n,t_1) \ldots DT_\gamma(t_2,t_3) DT_\gamma(t_1,t_2)\, ,$$
where the Jacobian matrix $DT_\gamma(t_i,t_{i+1})$ is given by formula (\ref{differential}). To simplify, we use the symbol $\prod_{j=1}^n DT_\gamma(t_j,t_{j+1})$ to indicate the above formula, aware that this product is clearly non-commutative. Moreover, see statement of Theorem \ref{T1}, recall that $\mathcal{T}_1 \subset \mathcal{S}$ denotes the residual set such that, for every $\gamma \in \mathcal{T}_1$, every periodic orbit in $\gamma$ is in ``general position'' (that is, no periodic orbit in $\gamma$ passes --within one period-- through the same point more than once).

\begin{theorem}\label{Franks}
For every $\gamma \in \mathcal{T}_1$, for every periodic orbit $\mathcal{O}$ for $T_\gamma$ of period at least $3$ and for every $C^2$ neighborhood $U$ of $\gamma$, there exists an open ball $B \subset Sp(1)$ centered in $DT_\gamma(\mathcal{O})$ such that any element of $B$ is realizable as $DT_{\tilde{\gamma}}(\mathcal{O})$ for some $\tilde{\gamma} \in U$. Moreover, such a perturbation can be supported in an arbitrary small neighborhood of three successive bouncing points of $\mathcal{O}$. 
\end{theorem}
\noindent In order to prove this theorem, we need to premise the next technical result (we refer to \cite[Lemma 7]{Viss} for the proof).
\begin{lemma} \label{lemma Viss}
For $\gamma \in \mathcal{S}$, let $\gamma(t) \in \Gamma$ with curvature $k(t)$. There exists $\varepsilon_\gamma > 0$ such that for any $0 < \varepsilon < \varepsilon_\gamma$, for any $|k - k(t)| < \varepsilon$ and for any neighbourhood $V$ of $\gamma(t)$ in $\Gamma$, there exists a perturbation $\tilde{\gamma}$ of $\gamma$ with the following properties:
\begin{itemize}
\item[1.] $k_{\tilde{\gamma}}(t) = k$,
\item[2.] $\tilde{\gamma} = \gamma$ outside $V$,
\item[3.] $\tilde{\gamma} \in \mathcal{S}$,
\item[4.] $\| \tilde{\gamma} - \gamma\|_{C^2} < \varepsilon$.
\end{itemize} 
\end{lemma}
\begin{proof}[Proof of Theorem~\ref{Franks}]
Let $\{(t_1,s_1), (t_2,s_2), \ldots, (t_n,s_n)\}$ be the coordinates of a $n$-periodic orbit $\mathcal{O}$ of $T_{\gamma}$, with $n \ge 3$. Correspondingly, for $j=1,\ldots,n$, let
\begin{equation} \label{jacobian}
A(t_j,t_{j+1}) :=
\begin{pmatrix}
 L_{11}(t_j,t_{j+1}) & 1\\
- L_{12}^2(t_j,t_{j+1})+ L_{22}(t_{j},t_{j+1})\cdot L_{11}(t_j,t_{j+1})&  L_{22}(t_j,t_{j+1})
    \end{pmatrix}\, .
\end{equation}
Notice --since $L$ and its derivates are $2\pi$-periodic and $\mathcal{O}$ is $n$-periodic-- that $A(t_n,t_{n+1}) = A(t_n,t_1)$. Then, by using formula (\ref{differential}), we obtain:
\begin{equation} \label{eccola}
\begin{split} 
DT_\gamma(\mathcal{O}) &= \prod_{j=1}^n DT_\gamma(t_j,t_{j+1}) = (-1)^n \prod_{j=1}^n \frac{A(t_j,t_{j+1})}{L_{12}(t_j,t_{j+1})} \, .
\end{split}
\end{equation}
In order to perform the perturbation, it is convenient to write $A(t_j,t_{j+1})$ as the product of the next two matrices: 
\begin{equation*}
A(t_j,t_{j+1}) :=
\begin{pmatrix}
 1 & 0\\
L_{22}(t_{j},t_{j+1}) &  1
    \end{pmatrix} 
\begin{pmatrix}
L_{11}(t_{j},t_{j+1}) & 1\\
-L_{12}^2(t_{j},t_{j+1}) &  0
    \end{pmatrix}\, , 
\end{equation*}
for $j=1,\ldots,n$. Notice that the first matrix contains the curvature $k(t_{j+1})$ in the term $L_{22}(t_{j},t_{j+1}) = \omega(\gamma(t_j),\gamma''(t_{j+1}))$ and the second one contains the curvature $k(t_{j})$ in the term $L_{11}(t_{j},t_{j+1}) = \omega(\gamma''(t_j),\gamma(t_{j+1}))$. Moreover, by a direct computation, we obtain that
$$\prod_{j=1}^n A(t_j,t_{j+1}) = 
\begin{pmatrix}
1 & 0\\
L_{22}(t_{n},t_{1}) &  1
    \end{pmatrix} 
\left[ \prod_{j=1}^n \begin{pmatrix}
A_j & 1 \\
B_j &  0
\end{pmatrix}
\right]
\begin{pmatrix}
1 & 0\\
L_{22}(t_{n},t_{1}) &  1
\end{pmatrix}^{-1}\, ,
$$
where:
$$\begin{cases}
A_{j-1} := L_{11}(t_j,t_{j+1}) + L_{22}(t_{j-1},t_j) = \omega(\gamma''(t_j),\gamma(t_{j+1}) - \gamma(t_{j-1}))\\
B_{j-1} := -L_{12}^2(t_j,t_{j+1}) = - [\omega(\gamma'(t_j),\gamma'(t_{j+1}))]^2
\end{cases}
$$
for $j = 2, \ldots,n+1$. In particular, since $\gamma''(t) = k(t)J\gamma'(t)$, we have that
$$A_{j-1} = k_jC_{j-1}\, ,$$ 
where:
$$k_j := k(t_j) \quad \text{and} \quad C_{j-1} := \omega(J\gamma'(t_j),\gamma(t_{j+1}) -\gamma(t_{j-1}))$$
for $j = 2, \ldots,n+1$. Notice that every $C_{j-1} \ne 0$ since $\gamma(t_{j+1}) -\gamma(t_{j-1})$ is parallel to $\gamma'(t_j)$ for the symplectic billiard law. \\
Without loss of generality, suppose now to perturb the curvature at the points $\gamma(t_1)$, $\gamma(t_2)$ and $\gamma(t_3)$ in $\Gamma$. We notice that here we use the hypothesis that $\gamma \in \mathcal{T}_1$, otherwise the perturbation we are going to perform may involve other points along the periodic orbit (to be precise, it would be sufficient that no periodic orbit in $\gamma$ passes --within one period-- through these three points more than once). We then define:
\begin{equation*}
\label{DTo}
M := \begin{pmatrix}
A_3 & 1\\
B_3 &  0
\end{pmatrix}
\begin{pmatrix}
A_2 & 1\\
B_2 &  0
\end{pmatrix}
 \begin{pmatrix}
A_1 & 1\\
B_1 &  0
\end{pmatrix}\, .
\end{equation*}
In order to make sure that these perturbations move $DT_\gamma(\mathcal{O})$ in 3 distinct directions in $Sp(1)$, it is sufficient to check that the matrices:
$$\frac{\partial}{\partial k_1}M, \quad \frac{\partial}{\partial k_2}M, \quad \frac{\partial}{\partial k_3}M$$
are linearly independent. It holds:
$$\frac{\partial}{\partial k_1}M = C_3 \begin{pmatrix}
A_1 A_2 + B_1 & A_2\\
0 &  0
\end{pmatrix}, \qquad 
\frac{\partial}{\partial k_2}M = C_1 \begin{pmatrix}
A_2 A_3 + B_2 & 0 \\
A_2B_3 &  0
\end{pmatrix}, \qquad 
\frac{\partial}{\partial k_3}M = C_2 \begin{pmatrix}
A_1 A_3 & A_3 \\
A_1B_3 &  B_3
\end{pmatrix}\, .
$$
In particular, the condition of linear independence is satisfied if the matrix
$$\begin{pmatrix}
A_1 A_2 + B_1 & A_2 & 0 & 0 \\
A_2A_3 + B_2 & 0 & A_2B_3 & 0 \\
A_1 A_3 &  A_3 & A_1B_3 & B_3
\end{pmatrix}
$$
has maximal rank ($=3$). This is the case if $A_2 B_3 \ne 0$, which is surely true: $A_2 \ne 0$ as a consequence of the symplectic billiard law, and $B_3 \ne 0$ for the twist condition. \\
Let $\mathcal{K} = \{ k = (k_1 , k_2 , k_3 )\} \cong \mathbb{R}^3$ be the space of curvatures at the points $\gamma(t_1), \gamma(t_2)$ and $\gamma(t_3)$ respectively, and let $\Phi: \mathcal{K} \to Sp(1)$ be the map that assigns to each $(k_1,k_2,k_3)$ the product $DT_\gamma(\mathcal{O})$ defined in (\ref{eccola}). By the arguments above and the inverse function theorem, $\Phi$ is a local diffeomorphism and the image of a neighbourhood of $k$ under $\Phi$ contains an open ball $B \subset Sp(1)$ centered in $DT_\gamma(\mathcal{O})$. In order to conclude the proof, it is sufficient to apply Lemma \ref{lemma Viss} to points $\gamma(t_1)$, $\gamma(t_2)$ and $\gamma(t_3)$ of the $n$-periodic orbit $\mathcal{O}$ of $T_\gamma$.
\end{proof} 
\begin{remark}
It is worth noting that the proof of Franks' lemma for Birkhoff billiards requires that the initial table has the property that there is no focusing along $3$-length segments of periodic orbits, we refer to Condition (3) in \cite[Definition 4]{Viss}. The author then defines a subset of convex planar billiards that has this property, and shows that it is a residual set. It is interesting to note that, in the symplectic billiard setting, such a condition is not required. 
\end{remark}
\indent As an application of Franks' lemma, in the sequel we show that symplectic billiards with stably hyperbolic periodic points have a hyperbolic set as the closure of the hyperbolic periodic points. In order to state this result precisely, we need to introduce some definitions. 
\begin{definition} [Hyperbolic set]
Let $\gamma\in \mathcal{S}$. A compact invariant subset $\Lambda$ of $\mathcal{P}_\gamma$ is hyperbolic of $T_\gamma$ if there exist a $DT_\gamma$-invariant continuous splitting of the tangent bundle $T_\Lambda \mathcal{P}_\gamma = E_\Lambda^s\oplus E_\Lambda^u$ and $m\in\mathbb{N}$ such that
$$
\|DT_\gamma^m(x)|_{E^s_x}\|\leq\frac12\quad\text{and}\quad \|DT_\gamma^{-m}(x)|_{E^u_x}\|\leq \frac12 \qquad \forall x\in\Lambda\, .
$$
\end{definition} 
\noindent Fix $p \geq 3$ and let
$$
\mathcal{H}_{\ge p} := \{\gamma\in\mathcal{S}\colon \text{all periodic points of $T_\gamma$ of period $\geq p$ are hyperbolic}\}.
$$
For $\gamma \in \mathcal{H}_{\ge p}$, let denote by $\mathrm{Per}_{\mathcal{H}}(T_\gamma)$ the set of hyperbolic periodic points of $T_\gamma$. Using Franks' lemma for symplectic billiards, we get the following result. 

\begin{theorem}\label{stably hyperbolicity}
$\overline{\mathrm{Per}_{\mathcal{H}}(T_\gamma)}$ is a hyperbolic set of $T_\gamma$ for every $\gamma\in \mathrm{int}_{C^2}(\mathcal{H}_{\ge p})$.
\end{theorem}
\begin{remark}
We remind that a similar theorem was proved for billiards in convex bodies \cite[Theorem 7.1]{BDDGT24}. This result uses the notion of the so-called stably hyperbolic families of periodic sequences, first introduced by R. Mañé and further developed in the symplectic setting by G. Contreras~\cite{Contreras10}. It is very likely that $\mathcal H_{\ge p} = \emptyset$. However, the answer to this question is unknown; namely, does every symplectic billiard have a non-hyperbolic periodic point of arbitrarily large period?    
\end{remark}
\noindent In order to prove Theorem~\ref{stably hyperbolicity}, following the main lines in \cite[Section 8]{Contreras10}, we first introduce the concept of stably hyperbolic families of periodic sequences. \\
A sequence $\mathcal{A} = \{ A_i\}_{i \in \mathbb{Z}}$, $A_i \in \mathrm{Sp}(1)$ is called periodic if there exists $n \ge 1$ such that $A_{i+n} = A_i$ for every $i \in \mathbb{Z}$. The smallest $n$ satisfying this property is called the period of the sequence, $\mathrm{per}(\mathcal{A}) = n$. A periodic sequence $\mathcal{A} = \{ A_i\}_{i \in \mathbb{Z}}$ is said to be hyperbolic if the product $A_{n-1} \cdots A_0$ is a hyperbolic matrix. \\
Let $\mathcal{F} = \{\mathcal{A}^\alpha\}_{\alpha\in A}$ be a family of periodic sequences (in $\mathrm{Sp}(1)$). We say that the family $\mathcal{F}$ is bounded if there exists $Q > 0$ such that $\|A^\alpha_i\|< Q$ for every $\alpha\in A$ and $i\in\mathbb{Z}$, where $\| \cdot \|$ denotes the operator norm on the space of square matrices. \\
Two families of periodic sequences $\mathcal{F}=\{\mathcal{A}^\alpha\}_{\alpha\in A}$ and $\widetilde{\mathcal{F}}=\{\widetilde{\mathcal{A}}^\alpha\}_{\alpha\in \tilde{A}}$ are said to be periodically equivalent if they have the same indexing set ($A = \tilde{A}$) and the same periods ($\mathrm{per}(\mathcal{A}^\alpha)=\mathrm{per}(\widetilde{\mathcal{A}}^\alpha)$ for every $\alpha\in A$). \\
Given two periodically equivalent families $\mathcal{F}$ and $\widetilde{\mathcal{F}}$,  we define the distance
$$
d(\mathcal{F},\widetilde{\mathcal{F}}):=
\sup\left\{\|A^\alpha_i-\tilde{A}^\alpha_i\|\colon \alpha\in A \text{ and } i\in\mathbb{Z}\right\}\, .
$$
Finally, a family $\mathcal{F} = \{\mathcal{A}^\alpha\}_{\alpha\in A}$ is called:
\begin{itemize} 
\item[$(i)$] \emph{hyperbolic} if every sequence $\mathcal{A}^\alpha$ in the family is hyperbolic;
\item[$(ii)$] \emph{uniformly hyperbolic} if it is hyperbolic with contraction and expansion rates uniform in $\alpha$.
\item[$(iii)$] \emph{stably hyperbolic} if there exists $\varepsilon>0$ such that every periodically equivalent family $\widetilde{\mathcal{F}}$ satisfying
$d(\mathcal{F},\widetilde{\mathcal{F}})<\varepsilon$ is also hyperbolic.
\end{itemize}
For further discussion on notions above, we refer to \cite[Section 8]{Contreras10}. These terminology was originally introduced by R. Mañ\'e in \cite[Pag 524]{Mane82}. In order to prove Theorem~\ref{stably hyperbolicity}, we will use --together with Franks' lemma-- the following theorem (see \cite[Theorem~8.1]{Contreras10}).  

\begin{theorem}\label{Contreras}
Any bounded stably hyperbolic family of periodic sequences is uniformly hyperbolic.
\end{theorem}

\begin{proof}[Proof of Theorem~\ref{stably hyperbolicity}] Let $\gamma\in \mathrm{int}_{C^2}(\mathcal{H}_{\ge p})$ and consider a hyperbolic periodic point $x\in\mathrm{Per}_{\mathcal{H}}(T_\gamma)$ of period $n$. For each $i=0,\ldots,n-1$ define
$A_i(x) := DT_\gamma(T_\gamma^i(x))$. Since the map $T_\gamma$ is area preserving, the matrices $A_i(x)$ belong to $\mathrm{Sp}(1)$. Let
$\mathcal{A}_x := \{A_i(x)\}_{i\in\mathbb{Z}}$ be the periodic sequence obtained by repeating these matrices periodically. Moreover, let 
$\mathcal{F} := \{\mathcal{A}_x \colon x \in \mathrm{Per}_{\mathcal{H}}(T_\gamma)\}$
be the family of all such periodic sequences. By compactness and continuity of the derivative of $T_\gamma$,  the family $\mathcal{F}$ is bounded. \\
Now, we claim that the family $\mathcal{F}$ is stably hyperbolic. Indeed, suppose by contradiction that $\mathcal{F}$ is not stably hyperbolic. Then, by definition, there exists a periodically equivalent family
$\widetilde{\mathcal{F}}$ arbitrarily close to $\mathcal{F}$
containing a sequence $\widetilde{\mathcal{A}}$ which is not hyperbolic. This means that there exists a periodic point $x \in \mathrm{Per}_{\mathcal{H}}(T_\gamma)$ of sufficiently large period, say $\geq p$,  such that the derivative sequence along the orbit can be perturbed arbitrarily slightly to produce a non-hyperbolic matrix. By Theorem~\ref{Franks}, such a perturbation of the derivative along the periodic orbit of $x$ can be realized by an arbitrarily $C^2$-small perturbation of the curve $\gamma$. In other words, there exists a curve $\tilde{\gamma}$ arbitrarily $C^2$-close to $\gamma$ for which the corresponding symplectic billiard map $T_{\tilde{\gamma}}$ has a non-hyperbolic periodic point. This contradicts the assumption that $\gamma\in\mathrm{int}_{C^2}(\mathcal{H}_{\ge p})$. Hence, the family $\mathcal{F}$ must be stably hyperbolic. \\
By Theorem~\ref{Contreras},  any bounded stably hyperbolic family
of periodic sequences is uniformly hyperbolic,  i.e. admits a uniform hyperbolic splitting.
Therefore, the derivative of $T_\gamma$ along every periodic points $x\in \mathrm{Per}_{\mathcal{H}}(T_\gamma)$ admits a uniform invariant splitting
$T_x\mathcal{P}_\gamma = E^s_x \oplus E^u_x$, with uniform contraction and expansion rates. Since the splitting depends continuously on the derivative along periodic
points (e.g., \cite[Proposition~6.4.4]{katok1997book}), it extends continuously to the closure $\overline{\mathrm{Per}_{\mathcal{H}}(T_\gamma)}$. The extension still satisfies the hyperbolicity estimates, hence $\overline{\mathrm{Per}_{\mathcal{H}}(T_\gamma)}$ is a hyperbolic set.
\end{proof}

\section{Elliptic islands for generic symplectic billiards} \label{Jose}

For $\gamma \in \mathcal{S}$, 
let $(t_1,s_1) \in \mathcal{P}_\gamma$ be an elliptic $n$-periodic point of $T_\gamma$, see Definition \ref{elliptic def}: this means that $DT^n_\gamma(t_1,s_1)$ has eigenvalues $e^{\pm i \theta}$ for some $\theta \in (0,2\pi) \setminus \{\pi\}$. Throughout this section we will assume that $(t_1,s_1)$ is non-resonant up to order 4, according to the next definition.

\begin{definition} \label{strongly resonant elliptic def} An elliptic n-periodic point $(t_1,s_1) \in \mathcal{P}_\gamma$ is non-resonant up to order 4 if $q \theta \notin 2\pi\mathbb{Z}$ for $q=2,3,4$. 
\end{definition}

\noindent In the study of stability of these points, one can use the so-called Birkhoff normal forms. This approach gives a canonical change of variables taking $T^n_\gamma$ into the form
$$(x,X) \mapsto R_{\Theta(r)}(x,X) + O(r^{4})\, ,$$
where $r^2 = x^2 + X^2$ and $R_{\Theta(r)}$ is the matrix of the counterclockwise rotation of angle 
$$\Theta(r) = \theta + \tau_1r^2$$ in the plane (see Theorem 2.12 in \cite{M73}). The leading coefficient $\tau_1$ is called first Birkhoff coefficient. If this coefficient is non-zero, Moser's twist theorem, see Theorem 2.13 in \cite{M73}, assures that $(t_1,s_1)$ is stable. In particular, it is proved that, in such a case, the map $T^n_\gamma$ has an invariant curve surrounding $(t_1,s_1)$ and the region bounded by the invariant curve is also invariant. \\

\noindent The following theorem, see Theorem 3 in \cite{BG10}, provides a sufficient condition for the first Birkhoff coefficient to be non-zero.
\begin{theorem}
\label{theorem BG}
Let $(0,0)$ be a non-resonant up to order 4 elliptic fixed point of a family of planar area-preserving maps $F_\varepsilon$ which are generated in a neighborhood of $(0,0)$ by $L_\varepsilon$ satisfying:
$$\partial_\varepsilon L_\varepsilon(x,X)|_{\varepsilon=0} = C\frac{x^4+X^4}{24} + O(x^5 + X^5), \quad \text{for some} \quad C\neq0
$$
and $\partial_{12}L_0(0,0)\neq0$. Then there exists $\varepsilon_0>0$ such that for every $\varepsilon\in(-\varepsilon_0,\varepsilon_0)\setminus\{0\}$ the origin $(0,0)$ has non-zero first Birkhoff coefficient so that it is a stable fixed point of $F_\varepsilon$.
\end{theorem}
\noindent In what follows, specifically in Theorem \ref{marzo}, we will prove that the above theorem can be applied to a generic subset of strictly convex billiard tables, assuring that elliptic periodic points are stable. To prove this fact, we need to premise the next technical Lemma \ref{62}. Recall that $\mathcal{T}_1 \subset \mathcal{S}$ denotes the residual set such that, for every $\gamma \in \mathcal{T}_1$, every periodic orbit in $\gamma$ is in ``general position'' (that is, no periodic orbit in $\gamma$ passes --within one period-- through the same point more than once). Moreover $\mathcal{T}_5 \subset \mathcal{S}$ denotes the residual set such that every periodic orbit in $\gamma \in \mathcal{T}_5$ has at least one non self-conjugate point. We refer to Theorems \ref{T1} and \ref{T4} respectively. 
\begin{lemma} \label{62} For $\gamma \in \mathcal{T}_1 \cap \mathcal{T}_5$, let $\mathcal{O} = \{(t_1,s_1), (t_2,s_2), \ldots, (t_n,s_n)\}$ be a $n$-periodic orbit for $T_\gamma$, with $t_1$ non self-conjugated. Let $I$ be an interval in $[0,2\pi)$ such that $t_1 \in I$ and $t_j \notin I$ for every $j \ne 1$. Given $t_1',t_{n+1}' \in I$, there exist $\tilde{\gamma}_\varepsilon$, arbitrarily small $C^k$-normal $\varepsilon$-perturbation of $\gamma$, $k \ge 2$, and $(t_j^\varepsilon)_{j=2}^n$ such that $t_j^\varepsilon \notin I$ for every $j = 2,\ldots,n$ and $(t_1', t_2^\varepsilon, \ldots, t_{n}^\varepsilon,t_{n+1}')$ gives an orbit segment for $\tilde{\gamma}_\varepsilon$.
\end{lemma}
\begin{proof} We first recall that $(t_1', t_2^\varepsilon, \ldots, t_{n}^\varepsilon,t_{n+1}')$ gives an orbit segment if
$$\begin{cases}
L_1(t_2^\varepsilon,t_3^\varepsilon) + L_2(t_1',t_2^\varepsilon) = 0 \\
L_1(t_j^\varepsilon,t_{j+1}^\varepsilon) + L_2(t_{j-1}^\varepsilon,t_j^\varepsilon) = 0 \qquad j=3,\ldots,n-1 \\
L_1(t_n^\varepsilon,t_{n+1}') + L_2(t_{n-1}^\varepsilon,t_n^\varepsilon) = 0
\end{cases}$$
Consequently, by the Implicit Function Theorem, $(t_j^\varepsilon)_{j=2}^n$ exist if the $(n-1) \times (n-1)$ matrix 
$$M = \begin{pmatrix}
B(2) & L_{12}(t_2,t_3) & 0 & 0 & \ldots & 0 \\
L_{12}(t_2,t_3) & B(3) & L_{12}(t_3,t_4) & 0 & \ldots & 0 \\
0 & L_{12}(t_3,t_4) & B(4) & L_{12}(t_4,t_5) & \ldots & 0 \\
\ldots & \ldots & \ldots & \ldots & \ldots & \ldots \\
0 & 0 & 0 & 0 & L_{12}(t_{n-1},t_n) & B(n)
\end{pmatrix}$$
is invertible, where
$$\begin{cases}
B(2) := L_{11}(t_2^\varepsilon,t_3^\varepsilon) + L_{22}(t_1',t_2^\varepsilon) \\
B(j) := L_{11}(t_j^\varepsilon,t_{j+1}^\varepsilon) + L_{22}(t_{j-1}^\varepsilon,t_j^\varepsilon) \qquad j=3,\ldots,n-1 \\
B(n) := L_{11}(t_n^\varepsilon,t_{n+1}') + L_{22}(t_{n-1}^\varepsilon,t_n^\varepsilon) = 0
\end{cases}$$
We notice that $M$ is invertible if and only if $t_1$ and $t_{n+1}$ are not conjugate. 
Equivalently, $M$ is invertible if and only if the $(1,2)$-entry of $DT_\gamma(\mathcal{O})$ is different from $0$. That is, see (\ref{eccola}):
$$\left(\prod_{j=1}^n A(t_j,t_{j+1})\right)_{12} = \left(\prod_{j=3}^n A(t_j,t_{j+1}) A(t_2,t_3) A(t_1,t_2) \right)_{12} \ne 0\, .$$
Let us denote 
$$\prod_{j=3}^n A(t_j,t_{j+1}) = \begin{pmatrix} a & b \\ c & d 
\end{pmatrix}\, ,$$
where $a,b,c$ and $d$ do not depend on $\gamma(t_2)$, $\gamma'(t_2)$ and $\gamma''(t_2)$. Then, by formula (\ref{jacobian}), the above condition on the $(1,2)$-entry of $DT_\gamma(\mathcal{O})$ reads:
\begin{equation} \label{condizione termine matrice}
a \cdot L_{11}(t_2,t_3) + b \cdot \left( -L_{12}^2(t_2,t_3) + L_{22}(t_2,t_3) \cdot L_{11} (t_2,t_3) \right) + L_{22}(t_1,t_2) \cdot \left( a + b \cdot L_{22}(t_2,t_3) \right)  \ne 0\, .   
\end{equation}
If this is the case, then $\tilde{\gamma}_\varepsilon = \gamma$ for every small $\varepsilon > 0$. Otherwise, if the expression in (\ref{condizione termine matrice}) is equal to zero, it is sufficient to notice that the curvature $k(t_2)$ is contained only in the terms $L_{11}(t_2,t_3)$ and $L_{22}(t_1,t_2)$ and then apply Lemma \ref{lemma Viss} to the point $\gamma(t_2)$ of the $n$-periodic orbit $\mathcal{O}$ of $T_\gamma$. This gives the arbitrarily small $C^2$-normal $\varepsilon$-perturbation of $\gamma$ and concludes the proof. 
\end{proof}
\noindent We are now ready to prove the main result of this section, essentaily based on Theorem \ref{theorem BG} and Lemma \ref{62}. We refer to Theorem 4 in \cite{BG10} for the corresponding result in the framework of Birkhoff billiards.  
\begin{theorem} \label{marzo}
Let $\gamma \in \mathcal{T}_1 \cap \mathcal{T}_5$ with $C^k$ boundary, $k \ge 5$. For every $n$-periodic orbit $\mathcal{O}$ for $T_\gamma$ of period at least $3$ and for every non-resonant up to order 4 elliptic $n$-periodic point $(t_1,s_1) \in \mathcal{O}$, there exists $\tilde{\gamma}_\varepsilon$, arbitrarily small $C^k$-normal $\varepsilon$-perturbation of $\gamma$ for $k \ge 5$, such that $(t_1,s_1)$ is a non-resonant up to order 4 elliptic $n$-periodic point for $T_{\tilde{\gamma}_\varepsilon}$ which is stable with non-zero Birkhoff coefficient. 
\end{theorem}
\begin{proof} Let $I$ be a small interval in $[0,2\pi)$ such that $t_1 \in I$ and $t_i \notin I$ for every $i \ne 1$. Let $\mu:\mathbb{S} \to \mathbb{R}$ be a $C^{\infty}$-function satisfying:
\begin{enumerate}
\item $\mu(t) = 0$ for every $t \notin I$.
\item $\mu(t_1) = \mu'(t_1) = 0 = \mu''(t_1) = \mu'''(t_1)$ and $\mu^{(4)}(t_1) = 1$.
\item $\|\mu\|_k$ small enough to guarantee strictly positive curvature. 
\end{enumerate}
We consider the following $\varepsilon$-family of $C^k$-normal $\varepsilon$-perturbations of $\gamma$:
$$\tilde{\gamma}_\varepsilon(t) := \gamma(t) + \varepsilon \mu(t){\textbf{n}}(t)\, .$$
Condition 2. guarantees that $(t_1,s_1)$ is a non-resonant up to order 4 elliptic $n$-periodic point also for $T_{\tilde{\gamma}_\varepsilon}$. In particular, by condition 2. we have that:
\begin{equation} \label{sviluppo gamma}
\tilde{\gamma}_\varepsilon(t) = \gamma(t) + \varepsilon \frac{(t-t_1)^4}{4!}{\bf{n}}(t_1) + O(|t-t_1|^5)\, .
\end{equation}
Since $\gamma \in \mathcal{T}_1 \cap \mathcal{T}_5$, Lemma \ref{62} applies assuring that, given $t_1', t_{n+1}' \in I$, for any $\varepsilon > 0$ there exists $(t_j^\varepsilon)_{j=2}^{n-1}$ such that every $t_j^\varepsilon \notin I$ and $(t_1',t_2^\varepsilon,\ldots,t_{n-1}^\varepsilon,t_n')$ gives an orbit segment for $T_{\tilde{\gamma}_\varepsilon}$, that is:
~\newline
\begin{equation} \label{piece of orbit}
\begin{cases}
\omega(\tilde{\gamma}'_\varepsilon(t_2^\varepsilon),\tilde{\gamma}_\varepsilon(t_3^\varepsilon) - \tilde{\gamma}_\varepsilon(t'_1)) = 0 \\
\omega(\tilde{\gamma}'_\varepsilon(t_j^\varepsilon),\tilde{\gamma}_\varepsilon(t_{j+1}^\varepsilon) - \tilde{\gamma}_\varepsilon(t_{j-1}^\varepsilon)) = 0, \qquad j=3,\ldots,n-2 \\
\omega(\tilde{\gamma}'_\varepsilon(t_{n-1}^\varepsilon),\tilde{\gamma}_\varepsilon(t'_n) - \tilde{\gamma}_\varepsilon(t_{n-2}^\varepsilon)) = 0
\end{cases}
\end{equation}
Let now use condition 1. to write the corresponding action, which is:
$$\omega_\varepsilon(t_1',t_2^\varepsilon,\ldots,t_{n-1}^\varepsilon,t_n') := \omega(\tilde{\gamma}_\varepsilon(t'_1),\gamma(t^\varepsilon_2)) + \sum_{j=2}^{n-1} \omega(\gamma(t^\varepsilon_j),\gamma(t^\varepsilon_{j+1})) + \omega(\gamma(t^\varepsilon_{n-1}),\tilde{\gamma}_\varepsilon(t'_n))\, .$$
In order to apply Theorem \ref{theorem BG} --assuring the stability of $(t_1,s_1)$-- we need to compute
$$\frac{\partial}{\partial \varepsilon}\omega_\varepsilon(t_1',t_2^\varepsilon,\ldots,t_{n-1}^\varepsilon,t_n')|_{\varepsilon = 0} =
\left[ \omega \left(\frac{\partial}{\partial \varepsilon} \tilde{{\gamma}}_\varepsilon(t_1'),\gamma(t_2^\varepsilon)\right) + \omega \left(\gamma(t_{n-1}^\varepsilon),\frac{\partial}{\partial \varepsilon} \tilde{{\gamma}}_\varepsilon(t_n')\right)\right]|_{\varepsilon = 0}$$
We stress that, in the equality above, all other terms cancel by (\ref{piece of orbit}). \\
Then, by (\ref{sviluppo gamma}), we get:
\begin{eqnarray*}
\omega \left(\frac{\partial}{\partial \varepsilon} \tilde{{\gamma}}_\varepsilon(t_1'),\gamma(t_2^\varepsilon)\right) &=& \frac{(t_1'-t_1)^4}{4!}\omega({\bf{n}}(t_1),\gamma(t^\varepsilon_2)) + O(|t_1'-t_1|^5) \\
&=& \frac{(t_1'-t_1)^4}{4!}\omega({\bf{n}}(t_1),\gamma(t_2)) + O((|t_1'-t_1| + |t_n' - t_1|)^5)\, ,
\end{eqnarray*}
where, in the last equality, we have used $\gamma(t^\varepsilon_2) = \gamma(t_2) + O(|t_1'-t_1| + |t_n' - t_1|)$. \\
Similarly:
\begin{eqnarray*}
\omega \left(\gamma(t_{n-1}^\varepsilon),\frac{\partial}{\partial \varepsilon} \tilde{{\gamma}}_\varepsilon(t_n')\right) &=& \frac{(t_n'-t_1)^4}{4!}\omega(\gamma(t^\varepsilon_{n-1}),{\bf{n}}(t_1)) + O(|t_n'-t_1|^5) \\
&=& \frac{(t_n'-t_1)^4}{4!}\omega(\gamma(t_n),{\bf{n}}(t_1)) + O((|t_1'-t_1| + |t_n' - t_1|)^5)\, .
\end{eqnarray*}
In conclusion, we have that
$$\frac{\partial}{\partial \varepsilon}\omega_\varepsilon(t_1',t_2^\varepsilon,\ldots,t_{n-1}^\varepsilon,t_n')|_{\varepsilon = 0} = A \frac{(t_1'-t_1)^4}{4!} + B \frac{(t_n'-t_1)^4}{4!} + O((|t_1'-t_1| + |t_n' - t_1|)^5)$$
where
$$A = A(t_1,t_2) = \omega({\bf{n}}(t_1),\gamma(t_2)) \quad \text{and} \quad B = B(t_1,t_n) = \omega(\gamma(t_n),{\bf{n}}(t_1))\, .$$
Since the symplectic billiard map does not depend on the choice of the origin --see point 3. of Proposition \ref{properties}-- let $O$ in the interior of the table be such that $\gamma(t_2) + \gamma(t_n) = 0$. This gives $A = B$, in fact:
$$A = \omega({\bf{n}}(t_1),\gamma(t_2)) = \omega(\overbrace{\gamma(t_n) + \gamma(t_2)}^{=0} - \gamma(t_2),{\bf{n}}(t_1)) = B\, .$$
With this choice:
$$\frac{\partial}{\partial \varepsilon}\omega_\varepsilon(t_1',t_2^\varepsilon,\ldots,t_{n-1}^\varepsilon,t_n')|_{\varepsilon = 0} = A \left( \frac{(t_1'-t_1)^4 + (t_n'-t_1)^4}{4!} \right) + O((|t_1'-t_1| + |t_n' - t_1|)^5)\, .$$
Notice that $A \ne 0$ since $2A = A + B = \omega({\bf{n}}(t_1),\gamma(t_2)-\gamma(t_n)) \ne 0$. Consequently, we can apply Theorem \ref{theorem BG} assuring that $(t_1, s_1)$ is a non-resonant up to order 4 elliptic $n$-periodic point
for $T_{\tilde{\gamma}_\varepsilon}$ with non-zero Birkhoff coefficient. In particular, by \cite[Theorem 2.13]{M73}, $(t_1, s_1)$ is stable. 
\end{proof}

\section{Transversality and positive topological entropy}\label{sec:transversality}

In this section we prove a perturbative result ensuring that heteroclinic
intersections between hyperbolic periodic orbits can be made transverse by an arbitrarily small normal perturbation of the symplectic billiard table. Similar results have been proved for classical billiards, e.g. \cite{Donnay2005,DCOKPC07} and --more recently-- \cite{FLORIO} in the analytic category. \\
~\newline
\noindent Let $\gamma \in \mathcal{S}$. Recall that the stable and unstable manifolds of a hyperbolic periodic point $x \in \mathcal{P}_\gamma$ are defined respectively by
\begin{eqnarray*}
W^{s}(x) := \{z\in \mathcal{P}_\gamma \colon d(T_\gamma^{n}(z),T_\gamma^{n}(x)) \to 0 \text{ as } n\to +\infty\}
\end{eqnarray*}
and
\begin{eqnarray*}
W^u(x) := \{z\in \mathcal{P}_\gamma \colon d(T_\gamma^{-n}(z),T_\gamma^{-n}(x))\to 0 \text{ as } n\to +\infty\}\, ,
\end{eqnarray*}
where $d$ denotes the distance function in $\mathcal{P}_\gamma$. Moreover, if $x \ne y$ are hyperbolic periodic points of $T_\gamma$, $z\in W^s(x)\cap W^u(y)$ is a heteroclinic point. The previous intersection is said to be transverse if $T_z W^s(x) \oplus T_z W^u(y) = \mathbb{R}^2$. By the stable manifold theorem, $W^s(x)$ and $W^u(x)$ are $C^1$-immersed submanifolds (curves in our case). In particular, in a sufficiently small neighborhood of a heteroclinic point $z$, the local stable and unstable manifolds are graphs of $C^1$-functions. In our specific case,  using the coordinates $(t,s)$, the local branches of the invariant manifolds can be written as
graphs $s = \varphi^{s,u}(t)$. Thus, the intersection is transverse if the derivatives of the corresponding
functions are different at the intersection point. The following proposition shows that a non-transverse heteroclinic intersection can be made transverse by an arbitrary small normal perturbation of the boundary.
\begin{proposition}\label{prop:transversality}
For $\gamma \in \mathcal{T}_1$, let $x, y$ be hyperbolic periodic points for $T_\gamma$ such that $W^s(x)\cap W^u(y) \ne \emptyset$. Then there exist a heteroclinic (homoclinic, if $x = y$) point $z$ of $T_{\gamma}$ and $\gamma_\varepsilon$, arbitrarily small $C^k$-normal $\varepsilon$-perturbation of $\gamma$, $k \ge 2$, such that:
\begin{enumerate}
\item $x$ and $y$ are hyperbolic periodic points for $T_{\gamma_\varepsilon}$,
\item $z$ is a heteroclinic (homoclinic, if $x = y$) point of $T_{\gamma_\varepsilon}$,
\item $T_z W^s(x) \oplus T_z W^u(y)=\mathbb{R}^2$ for $T_{\gamma_\varepsilon}$.
\end{enumerate}
\end{proposition}

\begin{proof}
Choose $z \in W^s(x)\cap W^u(y)$ such that the periodic orbits of $x$ and $y$ do not pass through the same point corresponding to $z$ on the boundary. Denote by $\mathcal{O}(z)=(z_j)_{j\in\mathbb{Z}}$ the orbit of the heteroclinic point $z_1:=z$ under $T_\gamma$, and write $z_j=(t_j,s_j)$. Since $\gamma \in \mathcal{T}_1$, there is an interval $I \subset [0,2\pi)$ such that $t_1 \in I$ and $t_j \notin I$ for $j \ne 1$. Moreover, by hypothesis, we may also assume that $I$ does not contain any point corresponding to the periodic orbits of $x$ and $y$. \\
We proceed to perform a perturbation supported within $I$, as in the proof of Theorem \ref{marzo}. Let $\mu:\mathbb{S} \to \mathbb{R}$ be a $C^{\infty}$-function satisfying:
\begin{enumerate}
\item $\mu(t) = 0$ for every $t \notin I$.
\item $\mu(t_1) = 0 = \mu'(t_1)$ and $\mu''(t_1) = 1$.
\item $\|\mu\|_k$ small enough to guarantee strictly positive curvature. 
\end{enumerate}
We consider the following $\varepsilon$-family of $C^k$-normal $\varepsilon$-perturbations of $\gamma$: $$\gamma_\varepsilon(t) := \gamma(t) + \varepsilon \mu(t){\textbf{n}}(t)\, .$$ 
We notice that, by second order contact at $t_1$, the periodic orbits $\mathcal{O}(x)$ and $\mathcal{O}(y)$ persist, are hyperbolic, and $(z_n)_{n\in\mathbb{Z}}$ remains a heteroclinic orbit for $T_{\gamma_\varepsilon}$.
Now, by the invariance of the stable manifold at the point $z$, we have:
$$
T_z W^s_{T_{\gamma_\varepsilon}}
(x) = DT_{\gamma_\varepsilon}^{-1}(T_{\gamma_\varepsilon}(z)) T_{T_{\gamma_\varepsilon}(z)} W^s_{T_{\gamma_\varepsilon}}
(T_{\gamma_\varepsilon}(x)) \, .
$$
Since $T_{\gamma_\varepsilon}(z)= T_{\gamma}(z)$ and $T_{\gamma_\varepsilon}(x)= T_{\gamma}(x)$, we have:
$$
T_{T_{\gamma_\varepsilon}(z)} W^s_{T_{\gamma_\varepsilon}}
(T_{\gamma_\varepsilon}(x)) = T_{T_{\gamma}(z)} W^s_{T_{\gamma}}
(T_{\gamma}(x)) = DT_\gamma(z) T_z W^s_{T_\gamma}
(x)\, ,
$$
so that, we may write
\begin{equation}\label{eq TWs}
T_z W^s_{T_{\gamma_\varepsilon}}
(x) = DT_{\gamma_\varepsilon}^{-1}(T_{\gamma}(z)) DT_\gamma(z) T_z W^s_{T_\gamma}
(x)\, .
\end{equation}
Similarly, for the unstable manifold at $z$, we have:
\begin{equation}\label{eq TWu}
T_z W^u_{T_{\gamma_\varepsilon}}(y) =DT_{\gamma_\varepsilon}(T^{-1}_{\gamma}(z)) DT_{\gamma}^{-1}(z) T_z W^u_{T_{\gamma}}(y)\, .
\end{equation}
In order to compare the tangent spaces $T_z W^s_{T_{\gamma_\varepsilon}}(x)$ and $T_z W^u_{T_{\gamma_\varepsilon}}(y)$, we describe how the derivative $DT_\gamma$ acts on directions via its induced projective action. A direction in $\mathbb{R}^2$ can be represented by its slope $m \in \overline{\mathbb{R}}$, that is, by the vector $(1,m)$. Under this identification, by formula \eqref{differential}, the action of $DT_\gamma$ on directions is described by its projective map:
$$
\widehat{DT_\gamma}(t,s)(m) = L_{22} - \frac{L_{12}^2}{L_{11}+m},\quad m\in \overline{\mathbb{R}}\, .
$$
Similarly:
$$
\widehat{DT_\gamma^{-1}}(t,s)(m) = -L_{11} + \frac{L_{12}^2}{L_{22}-m},\quad m\in \overline{\mathbb{R}}\, .
$$ 
Therefore, denoting by $m^{s}_\varepsilon$ (resp. $m^{u}_\varepsilon$) the slope of $T_z W^s_{T_{\gamma_\varepsilon}}(x)$ (resp. $T_z W^u_{T_{\gamma_\varepsilon}}(y)$), by \eqref{eq TWs} and \eqref{eq TWu} we have:
\begin{align*}
m^{s}_\varepsilon &= -\omega(\gamma_\varepsilon''(t_1),\gamma(t_2)) +\omega(\gamma''(t_1),\gamma(t_2)) + m^s_0 \\
m^{u}_\varepsilon &= \omega(\gamma(t_{0}),\gamma_\varepsilon''(t_1)) -\omega(\gamma(t_{0}),\gamma''(t_1)) + m^u_0\, .
\end{align*}
Thus,
$$
\frac{d}{d\varepsilon} (m_\varepsilon^s-m_\varepsilon^u)|_{\varepsilon = 0} = - \omega({\bf{n}}(t_1),\gamma(t_2)) - \omega(\gamma(t_{0}),{\bf{n}}(t_1)) = \omega( \gamma(t_2) - \gamma(t_{0}),{\bf{n}}(t_1)) \, ,
$$
which is always non-zero because (by the symplectic billiard law) $\gamma(t_2)-\gamma(t_{0}) \perp {\bf{n}}(t_1)$. This shows that, for $\varepsilon > 0$ sufficiently small, $m^{s}_\varepsilon\neq m^{u}_\varepsilon$, i.e. $T_z W^s(x) \oplus T_z W^u(y)=\mathbb{R}^2$ for $T_{\gamma_\varepsilon}$.
\end{proof}



\noindent For $n\geq3$, let now $\mathcal{R}_n \subset \mathcal{S}$ be the set of $\gamma \in \mathcal{S}$ such that:
\begin{enumerate}
\item[$(a)$] $\text{Per}_n(T_\gamma) := \{x\in \mathcal{P}_\gamma \colon T^n_\gamma (x) = x\}$ is finite and all its points are non-degenerate; 
\item[$(b)$] given any two periodic points $x,y\in \text{Per}_n(T_\gamma)$,  the unstable manifold $W^u(x)$ and the stable manifold $W^s(y)$ either do not intersect,  or else intersect at some transverse heteroclinic point;
\item[$(c)$] every elliptic periodic point $x \in \text{Per}_n(T_\gamma)$ is non-resonant up to order 4 and has non-zero first Birkhoff coefficient. In particular --as a consequence of  Moser’s twist theorem-- the map $T^n_\gamma$ has an invariant curve surrounding $x$ and the region bounded by the invariant curve is also invariant.
\end{enumerate}

\begin{proposition}\label{generic Rn}
For every $n\geq3$,  the set $\mathcal{R}_n$ is $C^5$-open and $C^\infty$-dense in $\mathcal{S}$. Consequently:
the set $\mathcal{R} := \bigcap_{n\geq3} \mathcal{R}_n$
is residual in $\mathcal{S}$.
\end{proposition}
\begin{proof}
Fix $n\ge 3$. That $\mathcal R_n$ is $C^5$-open follows from the continuity of the map $\gamma \mapsto T_\gamma$, together with the fact that --upon parametrizing $\gamma$ via its support function-- the symplectic billiard map $T_\gamma$ is a $C^\infty$ map on the cylinder $\mathbb S \times [0,\pi]$, and the persistence results for non-degenerate periodic points, transverse heteroclinic intersections, and elliptic points with non-vanishing first Birkhoff coefficient in conservative dynamical systems (see, for example, \cite[Theorem 1B, Theorem 3, Theorem 9]{Robinson70}). \\
\noindent It remains to prove that $\mathcal R_n$ is $C^\infty$-dense. Let $\gamma\in\mathcal S$ and let $U$ be a $C^\infty$-neighbourhood of $\gamma$. By Theorem \ref{Nondeg} and Corollary \ref{finite and nondeg}, after an arbitrarily small perturbation inside $U$, we may assume that all $n$-periodic points are non-degenerate and there are only finitely many periodic points. Hence, we get property (a). By Theorem \ref{nonresonant}, we may further assume that every elliptic $n$-periodic point is non-resonant (in particular, non-resonant up to order 4). Next, applying Theorem \ref{marzo} to each elliptic $n$-periodic orbit, and using the finiteness of $\mathrm{Per}_n(T_\gamma)$, we obtain an arbitrarily small perturbation inside $U$ for which every elliptic $n$-periodic point has non-zero first Birkhoff coefficient. Hence, property (c) holds. \\
\noindent To obtain property (b), we use Proposition \ref{prop:transversality}. Since there are only finitely many $n$-periodic points, only finitely many pairs of stable and unstable manifolds have to be considered. For each such pair of stable and unstable manifolds $W^s(x)$ and $W^u(y)$, either they do not intersect, or else they intersect at some heteroclinic point $z$. Since there are infinitely many heteroclinic points, we may assume without loss of generality that the periodic points do not pass through the same point corresponding to $z$ on the boundary of the billiard table. Therefore, by Proposition \ref{prop:transversality}, we obtain an arbitrarily small perturbation of $\gamma$ inside $U$ such that the intersection of perturbed stable and unstable manifolds at the heteroclinic point becomes transverse. Repeating this procedure for each of the finitely many pairs of hyperbolic $n$-periodic points, and choosing the perturbations successively with sufficiently small disjoint supports, we obtain a perturbation of $\gamma$ inside $U$ for which every heteroclinic intersection between hyperbolic $n$-periodic points is transverse. Hence, property (b) holds. \\
\noindent Since all of the above perturbations can be arbitrarily small and performed successively with disjoint supports, we obtain a billiard table $\tilde\gamma\in U$ that simultaneously satisfies (a)-(c). Therefore, $\mathcal R_n$ is $C^\infty$-dense in $\mathcal S$.
\end{proof}
\noindent The genericity of the positive topological entropy is consequence of next Theorem \ref{transverse homoclinic}, the proof of which requires the following result on existence of smooth invariant curves close to the boundaries of the phase space. This result has already been proved in \cite[Theorem 3]{AT}; we present a (slightly different) proof for completeness.
\begin{proposition} \label{tipo KAM}
For every $\gamma \in \mathcal{S}$, arbitrary close to the boundaries of the phase space $P_\gamma$, there exist smooth invariant curves for the symplectic billiard map $T_\gamma$.
\end{proposition}
\begin{proof} Choose as parameter the angle formed by the tangent direction with the vector $(1,0)$. With this choice, 
we have $\gamma'(\theta)=\rho(\theta)\mathbf{e}_\theta$, where $\rho(\theta)$ denotes the radius of curvature at the point $\gamma(\theta)$ and $\mathbf{e}_\theta$ is the unit tangent vector at $\gamma(\theta)$. Following previous convention, we will identify $\mathbb{S} \simeq \mathbb{R}/2\pi\mathbb{Z}$.
Using $\theta$ as the parameter $t$, when describing the phase space $\mathcal{P}_\gamma$, instead of using $s_1=-L_1(\theta_1,\theta_2)$ we use $\phi_1=\theta_2-\theta_1$.
 We obtain $0<\phi_1<\pi$, and the (open) phase space in these coordinates is the cylinder $\mathbb{S} \times (0,\pi)$. We now compute the Taylor expansion of $T_{\gamma}$ near the boundary components corresponding to $\phi_1=0$ and $\phi_1=\pi$. Since the behavior near $\phi_1=0$ is already well understood --see e.g. \cite[Theorem 3]{AT}-- we study $T_\gamma(\theta_1,\pi-\varepsilon)$ close to the other boundary $\phi_1=\pi$. Let $\varepsilon>0$ sufficiently small. By definition:
\[
T(\theta_1,\pi-\varepsilon)=(\theta_2,\phi_2)=(\theta_1+\pi-\varepsilon,\theta_3-\theta_2)\, ,
\]
where $\theta_3$ is determined by the condition
\[
\omega(\gamma'(\theta_2),\gamma(\theta_3)-\gamma(\theta_1))=0\, .
\]
Recall that as $\varepsilon\to 0^+$ we have $T(\theta_1,\pi-\varepsilon)\to (\theta_1+\pi,\pi^-)$, hence $\theta_3\to(\theta_1+2\pi)^-$. We therefore solve the defining equation for $\theta_3$ in a neighborhood of $\theta_1$, and then add $2\pi$ at the end. \\
\noindent We rewrite the equation as
\[
\omega(\gamma'(\theta_1+\pi-\varepsilon),\gamma(\theta_3)-\gamma(\theta_1))=0\, .
\]
To make the equation nondegenerate in $\theta_3$, we divide by $\theta_3-\theta_1$ and obtain
\[
\omega\!\left(\gamma'(\theta_1+\pi-\varepsilon),\frac{\gamma(\theta_3)-\gamma(\theta_1)}{\theta_3-\theta_1}\right)=0\, .
\]
Define:
\[
g(\theta_1,\theta_3):=\frac{\gamma(\theta_3)-\gamma(\theta_1)}{\theta_3-\theta_1}\, 
\]
and extend this by continuity via $g(\theta_1,\theta_1) = \gamma'(\theta_1)$.
Thus the equation for $\theta_3$ can be written as
\begin{equation}\label{theta3}
    f(\theta_1,\varepsilon,\theta_3):=\omega(\mathbf{e}_{\theta_1-\varepsilon},g(\theta_1,\theta_3))=0\, .
\end{equation}
As $\theta_3\to\theta_1$, we have
\[
g(\theta_1,\theta_3)=\gamma'(\theta_1)+\tfrac12\gamma''(\theta_1)(\theta_3-\theta_1)+O((\theta_3-\theta_1)^2),
\]
from which it follows that
\[
\partial_{\theta_3} f(\theta_1,0,\theta_1)
= \tfrac12\, \omega\bigl(\mathbf{e}_{\theta_1},\gamma''(\theta_1)\bigr)
= \tfrac12\, \omega\bigl(\mathbf{e}_{\theta_1},\rho(\theta_1)\mathbf{e}_{\theta_1+\tfrac{\pi}{2}}\bigr)
= \tfrac{\rho(\theta_1)}{2}\, .
\]
By the Implicit Function theorem, the relation \eqref{theta3} defines a 
smooth function $\theta_3(\theta_1,\varepsilon)$. Furthermore,
\[
\partial_\varepsilon f(\theta_1,0,\theta_1) = \rho(\theta_1)\, ,
\]
hence
\[
\partial_\varepsilon \theta_3(\theta_1,0)=-2\, .
\]
Therefore, applying Taylor’s theorem and then adding $2\pi$ to take into account the periodicity, we obtain
\[
\theta_3=\theta_1+2\pi-2\varepsilon+O(\varepsilon^2)\, ,
\]
so that
\[
\phi_2=\theta_3-\theta_2=\pi-\varepsilon+O(\varepsilon^2)\, .
\]
Consequently,
\[
T(\theta_1,\pi-\varepsilon)=\bigl(\theta_1+\pi-\varepsilon,\;\pi-\varepsilon+O(\varepsilon^2)\bigr)\, ,
\]
and 
--applying \cite[Theorem 2]{Russmann}-- the proof is
concluded.
\end{proof} 
\begin{remark}
We notice that, as an application of \cite[Theorem 2]{Russmann}, the previous proposition is valid as long as the regularity of the boundary is $C^k$, with $k > 5$.
\end{remark}
\begin{theorem}\label{transverse homoclinic}
For every $\gamma \in \mathcal{R}$, every hyperbolic periodic point of $T_\gamma$ has a transverse homoclinic point.
\end{theorem}
\begin{proof}
Let $\gamma\in \mathcal{R}$ and $x \in \mathcal{P}_\gamma$ be a hyperbolic periodic point of $T_\gamma$. We first notice that the stable and unstable manifolds of $x$ are contained in a compact region of the phase space. Indeed, $x$ is at a positive distance from the boundaries and --as a consequence of Proposition \ref{tipo KAM}-- there exist invariant curves accumulating on the boundaries. Since the stable and unstable manifolds of $x$ cannot cross any of these curves, they must
stay inside the region bounded by two invariant curves. In particular, they stay away from the boundaries. From this point on, the arguments of Z. Xia and P. Zhang~\cite{XZ14} apply verbatim. More precisely, Proposition 3.3 in \cite{XZ14} implies that every branch of the stable or unstable manifold of $x$ is either recurrent (i.e. either there exists $y\in W^u(x) $ such that $y \in \omega(T_\gamma,y)$ or there exists $z\in W^s(x)$ such that $z \in \alpha(T_\gamma,z)$) or it forms a path of saddle connections (i.e. branches of stable and unstable manifold coincide). This last case is excluded by the generic property $(b)$ of $\mathcal{R}$. Then, Proposition 4.2 in \cite{XZ14} yields the existence of an adjacent accumulating pair of stable and unstable branches. Applying Theorem 4.3 in \cite{XZ14}, we conclude that stable and unstable manifold intersect. Finally --again by the generic transversality property $(b)$ in $\mathcal R$-- $T_\gamma$ has a transverse homoclinic point associated to $x$.  
\end{proof}

\noindent Since a transverse homoclinic point of a hyperbolic periodic point gives rise, by the
Birkhoff--Smale homoclinic theorem, to a hyperbolic invariant set conjugate to a Bernoulli
shift on two symbols (Smale horseshoe), we obtain the following corollary.

\begin{corollary}\label{positive entropy}
For every $\gamma \in \mathcal{R}$, the symplectic billiard map $T_\gamma$ has positive topological entropy.
\end{corollary}

\begin{proof}
First of all,  notice that,  by Birkhoff theorem, $T_\gamma$ has an infinite number of periodic points. 
By Theorem~\ref{transverse homoclinic},  we just need to show that, among the infinite number of periodic points,  there is one that is hyperbolic.  Suppose,  by contradiction,  that all periodic points of $T_\gamma$ are non-hyperbolic and that $T_\gamma$ has zero topological entropy. Because $\gamma\in \mathcal R$, $T_{\gamma}$ has a stable elliptic periodic point $x$ with non-zero Birkhoff coefficient. Let $g := T_{\gamma}^n$ where $n \geq 3$ is the period of $x$. By Moser's twist theorem~\cite[Theorem 2.11]{M73}, in a neighborhood of $x$, there exists a Cantor family of invariant curves surrounding $x$. Choosing one such curve $\Upsilon$, the region it bounds is invariant under $g$. In suitable coordinates, this region, with the fixed point $x$ of $g$ removed, can be identified with an annulus on which $g$ acts as an area-preserving twist homeomorphism, preserving the boundary components. Let $\rho$ denote the rotation number of $g$ on $\Upsilon$; the rotation interval of $g$ on the annulus is then $[0,\rho]$. By a criterion due to Angenent~\cite[Theorem A]{Angenent1992}, if $T_\gamma$ has zero topological entropy, then $g$ has an essential invariant curve with rotation number $\alpha$ for every $\alpha \in [0,\rho]$. Fix a rational rotation number $\alpha = p/q$ in this interval. The invariant set with rotation number $\alpha$ is either a curve consisting entirely of $q$-periodic points, or a finite union of hyperbolic periodic points together with heteroclinic connections between them. The first case cannot clearly occur, since all periodic points on the curve would be degenerate, contradicting the fact that periodic points of $T_\gamma$ are non-degenerate. Therefore, we are left with the second case: $T_\gamma$ has a hyperbolic periodic point. This is also a contradiction and concludes the proof.
\end{proof}

\bibliographystyle{plain}
\bibliography{biblio.bib}

\end{document}